\documentclass[reqno,10pt]{amsart}
\usepackage{amsmath, latexsym, amsfonts, amssymb, amsthm, amscd}
\newtheorem {definition}{Definition}
\newtheorem{proposition}[definition]{Proposition}
\newtheorem {lemma}[definition]{Lemma}
\newtheorem {example}[definition]{Example}

\newtheorem*{problem}{Problem}

\newtheorem {theorem}[definition]{Theorem}
\newtheorem {rem}[definition]{Remark}

\author{Leif D\"oring}
\thanks{The author was supported by the Fondation Science Mat\'ematiques de Paris}
\address{Leif D\"oring, Laboratoire de Probabilit\'es et Mod\'eles Al\'eatoires Universit\'e Paris 6, 4 place Jussieu, 75252 Paris Cedex 05, France}
\email{leif.doering@upmc.fr}

\numberwithin{definition}{section}
\numberwithin{equation}{section}

\newtheorem {cor}[definition]{Corollary}

\newtheorem*{notation}{Notations}

\newcommand{\cA}{{\mathcal A}}
\newcommand{\cB}{{\mathcal B}}

\newcommand{\cD}{{\mathcal D}}

\newcommand{\cG}{{\mathcal G}}
\newcommand{\cH}{{\mathcal H}}
\newcommand{\cN}{{\mathcal N}}

\newcommand{\cP}{{\mathcal P}}

\newcommand{\D}{\mathbb{D}}
\newcommand{\E}{{E}}
\newcommand{\N}{\mathbb{N}}

\newcommand{\R}{\mathbb{R}}
\newcommand{\eps}{\varepsilon}
\newcommand{\1}{\mathbf{1}}
\DeclareMathOperator{\sign}{sign}
\newcommand{\dd}{d}

\title{A Jump-Type SDE Approach to Real-Valued Self-Similar Markov Processes}
\begin{document}
\maketitle

\begin{abstract}
	In his 1972 paper, John Lamperti characterized all positive self-similar Markov processes as time-changes of exponentials of L\'evy processes. In the past decade the problem of classifying all non-negative self-similar Markov processes 		that do not necessarily have zero as a trap has been solved gradually via connections to ladder height processes and excursion theory.\\
	Motivated by the recent article \cite{CPR}, we classify via jump-type SDEs the symmetric real-valued self-similar Markov processes that only decrease the absolute value by jumps and leave zero continuously.\\
	 Our construction of these self-similar processes involves a pseudo excursion construction and singular stochastic calculus arguments ensuring that solutions to the SDEs spend zero time at zero to avoid problems caused by a "bang-bang" drift.
\end{abstract}
\section{Introduction and Main Results}

\subsection{The Classification Problem for Self-Similar Markov Processes}
Dating back to Lamperti's seminal article \cite{L}, the study of self-similar Markov processes (originally called semi-stable processes by Lamperti) with values in a subset $E$ of $\R$ has attracted a lot of attention. In what follows, we will only discuss self-similar Markov processes in $E=\R$ and $E=[0,\infty)$ and denote by $\D$ the space of c\`{a}dl\`{a}g functions $\omega:\R_+\to E$ (right continuous with left limits) endowed with the Borel sigma-field $\cD$ generated by Skorokhod's topology. A strong Markov family $(P^z)_{z\in E}$ on $(\D,\cD)$ is called self-similar if the coordinate
 process $Z_t(\omega):=\omega(t), t\geq 0,$ fulfills
 the following scaling property: 
 \begin{align}\label{self_sim}
   \text{the law of $(c^{-1}Z_{ct})_{t\geq 0}$ under $P^z$ is $P^{c^{-1}z}$}
 \end{align}
 for all $c>0$ and $z\in E$.
 Typically, a more general self-similarity definition is given, replacing the power $-1$ by $-a$ for some $a>0$. In this treatment we fix without loss of generality the index of self-similarity $a=1$ since the change between index $a$ and $1$ can be performed by taking the power $Z^{1/a}$. 
 We will say that $Z$ (or alternatively the law $(P^z)_{z\in E}$) is a
 \begin{itemize}
	 \item positive self-similar Markov process  if $E=[0,\infty)$ and $Z$ is trapped at zero,
	 \item non-negative self-similar Markov process if $E=[0,\infty)$,
 	 \item $\R\backslash \{0\}=:\R_*$-valued self-similar Markov process if $E=\R$ and $Z$ is trapped at zero,
	 \item real-valued self-similar Markov process if $E=\R$.
 \end{itemize}
According to this definition, a positive self-similar Markov process is not really a positive process but the above classification seems to be the most rigorous to separate the appearing cases. Note that, if  $T_0=\inf\{t\geq 0:Z_t=0\}$ denotes the first hitting time of zero and $(Z^\dag_t)_{t\geq 0}:=(Z_{t\wedge T_0})_{t\geq 0}$ the process obtained from $Z$ by absorption at $0$, then in our notation a non-negative self-similar Markov process contains a positive self-similar Markov process and analogously a real- contains an $\R_*$-valued self-similar Markov process.
 \smallskip
 
We are interested in the following problem:
 \begin{problem}
	Classify all positive, non-negative, $\R_*$- and real-valued self-similar Markov processes.
 \end{problem}
The first three instances of the problem have been resolved and only the last remains open. 
\subsubsection{(I) Lamperti's Classification of Positive Self-Similar Markov Processes}
The fundamental result in the classification theory of self-similar Markov processes is Lamperti's representation obtained in \cite{L}. Lamperti showed that there is a bijection between positive self-similar Markov processes and L\'evy processes, possibly killed at an independent exponential time $\zeta$. For a L\'evy process $\xi$, Lamperti's representation of positive self-similar Markov processes takes the form
 \begin{align}\label{LT}
   Z_t= z\exp\big(\xi_{\tau({tz^{-1})}} \big),\qquad 0\leq t<T_0,
 \end{align}
  where the random time-change is given by the generalized inverse of the exponential functional of $\xi$, that is
	\begin{align*}
		\tau(t):=\inf\bigg\{s\geq 0:\int_0^s \exp\left(\xi_r\right) dr>t\bigg\}.
	\end{align*}
It is important to add that $T_0$ is finite almost surely for all initial conditions $z>0$ precisely if $\xi$ drifts to $-\infty$ and in this case $\tau(T_0z^{-1})=\infty$. Consequently, if we suppose that $\xi$ is set to $-\infty$ at the killing time $\zeta$, then Equation (\ref{LT}) is equally valid for $t\geq 0$ with $Z_t=0$ for any $t\geq T_0$. A consequence of Lamperti's representation is that the Feller property holds on $(0,\infty)$ for any non-negative self-similar strong Markov process.

\begin{example}\label{example}
	The only positive self-similar Markov processes (with self-similarity index $a=1$) with continuous sample paths are solutions to the stochastic differential equations (SDEs)
 \begin{align}\label{sb}
		 Z_t = a\,d t+\sigma \sqrt{Z_t}\,d B_t,\quad t\leq T_0,
 \end{align}
for $a\in\R, \sigma\geq 0$. Their Lamperti transformed L\'evy processes are $\xi_t=\big(a-\frac{\sigma^2}{2}\big)t + \sigma B_t$.
\end{example}
Lamperti's representation was successfully applied for the study of stable L\'evy processes since it allows via the identity $T_0\stackrel{\mathcal L}{=}\int_0^\infty \exp(\xi_r)\,dr$ the study of the exponential functional of a L\'evy process via the first hitting time of a self-similar process. For the use of Lamperti's transformation for the study of the maximum of stable L\'evy processes we refer for instance to Patie \cite{P} or Kutznetsov and Pardo \cite{KP} and references therein.

\subsubsection{(II) Classification of Non-Negative Self-Similar Markov Processes}
We mentioned above that a non-negative self-similar Markov process $Z$ contains a unique positive self-similar Markov process $Z^\dag$ by absorbtion at zero. As a consequence, the classification problem for non-negative self-similar Markov processes is equivalent to finding all self-similar extensions of positive self-similar Markov processes. The task has been resolved in recent years; first, if the corresponding L\'evy process $\xi$ drifts to $-\infty$ (i.e. $T_0<\infty$ a.s.), and later if $\xi$ fluctuates or drifts to $+\infty$ (i.e. $T_0=\infty$ a.s.).
\smallskip

It has first been proved independently by Rivero \cite{R1}, \cite{R2} and Fitzsimmons \cite{F} that positive self-similar Markov processes that hit zero in finite time can be extended uniquely to a non-negative self-similar Markov process that leaves zero continuously if and only if the Cram\'er type condition 
 \begin{align}\label{cramer}
		\text{there is a }0<\theta<1\text{ such that }\Psi(\theta)=0
 \end{align}
 holds. Here, and in what follows, whenever well-defined 
 \begin{align*}
 	\Psi(\theta)=\log \E\big(e^{\theta \xi_1}; \zeta>1\big),\quad \theta\geq 0,
\end{align*}
 denotes the Laplace exponent of the L\'evy process $\xi$ (killed at $\zeta$) that occurs in Lamperti's representation. The proofs of Rivero and Fitzsimmons are based on Blumenthal's general theory of Markov extensions developed in \cite{Bl}.
 \smallskip

For positive self-similar Markov processes not hitting zero in finite time, the classification problem is to give conditions when (and how) the family $(P^z)_{z>0}$ can be extended continuously to $z=0$ so that the extended process remains self-similar and leaves zero. This challenging question was answered subsequently in Bertoin and Caballero \cite{BC}, Bertoin and Yor \cite{BY},
Caballero and Chaumont \cite{CC} and Chaumont et al. \cite{CKPR}: The law $(P^z)_{z>0}$ of a positive self-similar Markov process extends continuously to initial condition $z=0$ if and only if
\begin{align}\label{over}
	\text{ the overshoot process }(\xi_{T_x}-x)_{x\geq 0} \text{ converges weakly as }x\to\infty.
\end{align}	
	If the overshoot process does not converge, then the laws $P^z$ converge, as $z\to 0$, to the degenerate law concentrated at $+\infty$.\\ The necessity of Condition \eqref{over} is relatively easy to prove and the main difficulty lies in the construction of $P^0$ if Condition \eqref{over} is valid. For the simplest construction of the non-degenerate limiting law $P^0$ via L\'evy processes started from $-\infty$, 	we refer to Bertoin and Savov \cite{BS}. A jump-type SDE approach which motivated the present article was developed in \cite{DB}.
	
\subsection{(III) Classification of $\R_*$-Valued Self-Similar Markov Processes}
	In contrast to self-similar Markov processes with non-negative sample paths, less is known about the classification of self-similar Markov processes with real-valued sample paths. The analogue to Lamperti's representation, called Lamperti-Kiu representation, has recently been proved in Chaumont et al. \cite{CPR} completing earlier work of Kiu \cite{K1}, \cite{K2} and Chybiryakov \cite{C}. To the best knowledge of the author the classification of real-valued self-similar Markov processes that leave zero remains open.
	\smallskip	

	The main idea of the Lamperti-Kiu representation is as follows: due to the assumed c\`{a}dl\`{a}g property of sample paths, the times $H_n$ of the $n$-th change of sign 
	\begin{align}\label{ChangeTime}
		H_0=0,\quad H_n=\inf\{t>H_{n-1}: Z_t Z_{t-}<0\},\quad n\geq 1,
	\end{align}
	can only accumulate at $T_0$. In the random intervals $[H_n,H_{n+1})$ the real-valued self-similar Markov process reduces to a 		strictly positive or strictly negative self-similar Markov process to which Lamperti's transformation can be applied and leads to two (possibly different) sequences $\xi^{+,n}$ and $\xi^{-,n}$ of L\'evy processes. Using the strong Markov property of $Z$, independence of the sequence $\xi^{\pm,n}$ follows so that the 		Lamperti-Kiu representation is obtained by glueing a sequence of Lamperti representations. A crucial additional ingredient are jumps $\Delta Z_{H_n}$ that determine the random initial condition for the positive/negative self-similar Markov processes on $[H_n,H_{n+1})$.
	 Again by the strong Markov property it was shown that those jumps are independent and  the rate of their occurrence is determined by a random time-change as in Lamperti's representation \eqref{LT}. Loosely speaking, the time-change accelerates all jumps with a rate $1/|Z_{s-}|$ and, consequently, $Z$ changes sign infinitely often before and after touching zero. The jumps $\Delta Z_{H_n}$  add many difficulties to the classification and prevent a straight forward adaption of arguments developed for positive self-similar Markov processes.	 
	\smallskip
	
Our main results are for symmetric real-valued self-similar Markov processes, that is the law of $-Z$ under $P^z$ is $P^{-z}$. In the symmetric case the Lamperti-Kiu transformed L\'evy processes satisfy $\xi^{+,n}\stackrel{\mathcal L}{=}\xi^{-,n}$.
	\smallskip
	
	A formal description of the Lamperti-Kiu representation is rather unpleasant since the change of sign is coded in the underlying L\'evy process via an additional complex direction.  We follow a different approach based on jump-type SDEs that seems to be more tracktable.
\begin{notation}
Solutions to SDEs are always considered on a stochastic basis $(\Omega,\cG,(\cG_t)_{t\geq 0},P)$ that is rich enough to carry all appearing Brownian motions and Poisson point processes and satisfies the usual conditions.
All SDEs are driven by a $(\cG_t)$-standard Wiener process $B$ and an independent $(\cG_t)$-Poisson random measure $\mathcal N$. We will use {weak solutions}, i.e. $(\mathcal G_t)_{t\geq 0}$ adapted stochastic processes $(Z_t)_{t\geq 0}$ with almost surely c\`{a}dl\`{a}g sample paths that satisfy an SDE in integrated form almost surely.  If additionally $Z$ is adapted to the augmented filtration generated by $B$ and $\mathcal N$, then $Z$ is said to be a {strong solution}. Pathwise uniqueness holds if for any two weak solutions defined on the same probability space with the same standard Wiener process and Poisson random measure, they are indistinguishable. In several SDEs the sign-function
		\begin{align*}
			\sign(x):=\1_{\{x>0\}}-\1_{\{x\leq 0\}}
		\end{align*}
		is used. We say a stochastic process $Z$ does not spend time at zero if almost surely
		\begin{align*}
			\int_0^\infty \1_{\{Z_s=0\}}\,ds=0.
		\end{align*}
\end{notation}
Here is a reformulation of the main result of \cite{CPR} via jump-type SDEs which we only state for symmetric $\R_*$-valued self-similar Markov processes with the additional assumption
\begin{align*}
		\textbf{(A)}\quad P^z( |Z_s|\leq |Z_{s-}|,\,\, \forall s\geq 0)=1,\quad z\in\R_*.
	\end{align*}
	Assumption \textbf{(A)} excludes the possibility that jumps of $Z$ increase the absolute value or, equivalently, that the L\'evy processes $\xi^{+,n}\stackrel{\mathcal L}{=}\xi^{-,n}$ of the Lamperti-Kiu representation have positive jumps. A general non-symmetric version without Assumption \textbf{(A)} is given below in Proposition \ref{PropositionMain2}.

\begin{proposition}\label{PropositionMain}
	\textbf{(I)} There is a bijection between symmetric $\R_*$-valued self-similar Markov processes satisfying Assumption \textbf{(A)} and quintuples $(a,\sigma^2,\Pi,q,V)$ consisting of
	\begin{itemize}
		\item a triplet $(a,\sigma^2,\Pi)$ of a spectrally negative L\'evy process killed at rate $q$ with Laplace exponent $\Psi$,
		\item  a finite measure $V(du)$ on $[-1,0)$.
	\end{itemize}
	\textbf{(II)} For a quintuple $(a,\sigma^2,\Pi,q,V)$ a symmetric real-valued self-similar Markov processes issued from $z\in\R_*$ can be constructed as strong solution to
			\begin{align}\label{Equation}
		\begin{split}
			Z_t&=z+\bigg(\Psi(1)+\int_{-1}^0 (u-1) V(du)\bigg)\int_0^t\sign(Z_s)\,ds+\sigma\int_0^t \sqrt{|Z_s|}\,dB_s\\
			&\quad+\int_0^t \int_0^{\frac{1}{|Z_{s-}|}}\int_{-1}^1 Z_{s-}(u-1)(\mathcal{N-N'})(ds,dr,du)
		\end{split}
	\end{align}
	for $t\leq T_0$.
	Here, $B$ is a standard Brownian motion and $\mathcal N$ is a Poisson point process on $(0,\infty)\times (0,\infty)\times [-1,1]$ with intensity measure $\mathcal N'(ds,dr,du)=ds\otimes dr\otimes \bar\Pi(du)$ according to the piecewise 					definition
		\begin{itemize}
			\item $\bar\Pi_{\big|(0,1]}(du)$ is the image measure under $\R_{-}\ni u\mapsto e^u$ of $\Pi$,
			\item $\bar\Pi(\{0\})=q$,
			\item $\bar\Pi_{\big|[-1,0)}(du)=V(du)$.
		\end{itemize}
\end{proposition}
	Let us briefly explain the ingredients of the jump-type SDE \eqref{Equation}. Comparing with Example \ref{example} the so-called "bang-bang" drift and the Brownian part might be not surprising since in the intervals $[H_n,H_{n+1})$, the restrictions of $Z$ (resp. -$Z$) have to be positive self-similar Markov processes. The jumps of  the Poissonian integral are such that
	\begin{align}\label{form}
		Z_{s-}\mapsto Z_{s-}+Z_{s-}(u-1)=Z_{s-} u,
	\end{align}
	and $u$ is chosen according to the measure $\bar\Pi$ which looks a bit complicated. We chose this formulation since it allows us to explain the three occurring jump possibilities for self-similar Markov processes with only one stochastic integral:
	\begin{itemize}
		\item	If $u>0$, then $Z$ does not change sign and consequently these are the jumps corresponding to a piecewise Lamperti transformation in $[H_n,H_{n+1})$. If the L\'evy measure $\Pi$ is infinite, then also $\bar\Pi$ is infinite with a possible pole only at $+1$ so that small jumps (i.e. $\Delta Z_s\approx 0$) accumulate.
		\item If $u=0$, then $Z$ jumps to zero which is equivalent to a jump to $-\infty$ (killing) for the L\'evy process in Lamperti's representation \eqref{LT}.
		\item If $u<0$, then $Z$ changes sign and the jump-times are precisely the $H_n$ from \eqref{ChangeTime}. The finiteness of the intensity measure $V(du)$ is equivalent to the non-accumulation of $H_n$ away from $T_0$.
	\end{itemize}
	The $dr$-integral is included to dynamically accelerated the jump rate by $1/|Z_{s-}|$. Hence, on the zero set of solutions all jumps come with infinite rate and the jumps not changing sign even with "double-infinite" rate if $\Pi$ is infinite. Such explosions of the jump rate  are the main difficulty of the SDE \eqref{Equation} when studied for all $t\geq 0$ or issued from $z=0$.

	\smallskip
	
		\begin{definition}\label{defi}
		\textbf{(a)} For a symmetric $\R_*$-valued self-similar Markov process the quintuple $(a,\sigma^2,\Pi,q,V)$ appearing in Proposition \ref{PropositionMain} (or appearing equivalently in the Lamperti-Kiu representation of \cite{CPR}) is called the corresponding Lamperti-Kiu quintuple.\\
		\textbf{(b)} A quintuple  $(a,\sigma^2,\Pi,q,V)$ is called the Lamperti-Kiu quintuple of a symmetric real-valued self-similar Markov process if it is the Lamperti-Kiu quintuple for the $\R_*$-valued self-similar Markov process obtained by absorption at zero.		
	\end{definition}
	

\subsection{Main Result}
	The striking feature of the SDE \eqref{Equation} compared to the time-change Lamperti-Kiu representation is that the form of possible extensions after hitting zero can be guessed immediately. If possible, they should be solutions to the same SDE for all $t\geq 0$. Here is the main result of this article:

\begin{theorem}\label{TheoremMain}
	\textbf{(I)} There is a bijection between symmetric real-valued self-similar Markov processes that leave zero continuously and satisfy Assumption \textbf{(A)} and quintuples $(a,\sigma^2,\Pi,q,V)$ consisting of
	\begin{itemize}
		\item a triplet $(a,\sigma^2,\Pi)$ of a spectrally negative L\'evy process killed at rate $q$ with Laplace exponent $\Psi$,
		\item  a finite measure $V(du)$ on $[-1,0)$
	\end{itemize}
	that satisfy 
	\begin{align}\label{Condition}
		\Psi(1)+\int_{-1}^0 (|u|-1) V(du)>0.
	\end{align}
	\smallskip
	
	\textbf{(II)} For a quintuple $(a,\sigma^2,\Pi,q,V)$ as in {(I)} a symmetric real-valued self-similar Markov processes that leaves zero continuously and is issued from $z\in\R$ can be constructed as weak solution to the SDE \eqref{Equation} for $t\geq 0$.

	\end{theorem}

The necessity of Condition \eqref{Condition} can be found easily by a reduction to Condition \eqref{cramer} for positive self-similar Markov processes.
The difficult part of the proof is a construction of a real-valued self-similar Markov process with Lamperti-Kiu quintuple $(a,\sigma^2,\Pi,q,V)$ whenever Condition \eqref{Condition} is valid. Our reformulation of the Lamperti-Kiu representation given in Proposition \ref{PropositionMain} turns out to be useful since it gives flexibility for the construction via approximation procedures and semimartingale calculus. The approximation is rather non-standard (and might remind the reader to constructions of skew Brownian motion) since the non-continuity of the "bang-bang" drift causes problems in weak convergence arguments. Limiting points of the approximating sequences might become trivial (trapped at zero) and it is precisely Condition \eqref{Condition} that ensures this is not the case.
\begin{rem}
It is surprising that Condition \eqref{Condition} is sufficient for the existence of solutions to the SDE \eqref{Equation} that leave zero. Since $\Psi$ does not depend on $V$, the quintuple $(a,\sigma^2,\Pi,q,V)$ can be chosen such that
	\begin{align*}
		\bigg(\Psi(1)+\int_{-1}^0 (u-1) V(du)\bigg)<0<\bigg(\Psi(1)+\int_{-1}^0 (|u|-1) V(du)\bigg).
	\end{align*}
	Then the SDE \eqref{Equation} has martingale terms that vanish at zero and a drift that points towards the origin. In such a situation it is impossible to find non-negative solutions to SDEs since positive martingales are absorbed at zero. Real-valued 		solutions, however, can exists.	Precisely the jumps crossing the origin cause this effect; if $V=0$ solutions can leave if and only if the drift points away from the origin.

	\end{rem}
It is important to note that the SDE \eqref{Equation} behaves very differently at zero and away from zero. Away from zero the coefficients are locally Lipschitz continuous so that pathwise uniqueness holds and strong solutions exist. Only when solutions touch zero the drift and the jumps are problematic. Consequently, the main task of the proofs is to give a construction and uniqueness statement for solutions issued from zero.

\subsection{Connection to Other Self-Similar SDEs}
Theorem \ref{TheoremMain} extends the results of \cite{DB} obtained for positive self-similar Markov processes with an assumption similar to Assumption \textbf{(A)}. The techniques utilized here need to be different from those of \cite{DB} since the drift coefficient is discontinuous so that standard arguments for SDEs do not apply. In particular, solutions to \eqref{Equation} are constructed even if the drift points towards zero which forces us to leave classical arguments in the spirit of Yamada and Watanabe and combine more specific stochastic calculus arguments with general martingale problem techniques.
The main result of \cite{DB} was stronger in the sense that pathwise uniqueness could be proved for their SDEs and solutions are automatically strong. Consequently, the constructed non-negative self-similar Markov processes are deterministic functionals of a Brownian motion and a Poisson point process so that we can speak of a strong classification.

\begin{rem}
	Possible uniqueness statements for the SDE \eqref{Equation} issued from $z=0$ need to be in a restricted sense as one can see at the simplest special case
	\begin{align}\label{bbbd}
		dZ_t= \sign(Z_t)\,dt+2 \sqrt{|Z_t|}\,dB_t,\quad Z_0=0.
	\end{align}
	Given a Brownian motion $W$, then three weak solution to the SDE \eqref{bbbd} can be defined explicitly:
	\begin{align*}
		Z^{(1)}_t=\sign(W_t) W_t^2,\quad  Z^{(2)}_t=W_t^2,\quad Z^{(3)}_t=-W_t^2.
	\end{align*}
	Of course, already the one-dimensional marginal distributions differ for the $Z^{(i)}$. Nonetheless, restricted to symmetric solutions (this rules out $Z^{(2)}$ and $Z^{(3)}$) one can easily deduce the uniqueness for the one-dimensional marginals. Tanaka's formula applied to \eqref{bbbd} shows that the absolute value of any solution satisfies
	\begin{align}\label{ddddd}
		X_t=t+2\int_0^t \sqrt{X_s}\,dB_s.
	\end{align}	
	Now uniqueness for \eqref{ddddd} implies uniqueness for the absolute value of solutions to \eqref{bbbd}, hence, uniqueness for one-dimensional marginals of symmetric solutions. \\
	The simple example shows that the best possible uniqueness statement for solutions to the SDE \eqref{Equation} is pathwise uniqueness among symmetric solutions.
\end{rem}

 There might be more sophisticated arguments that yield pathwise uniqueness among symmetric solutions, such as the arguments developed in Bass et al. \cite{BBC} for the self-similar SDE
\begin{align}\label{bass}
	dZ_t=|Z_t|^\beta\,dB_t, \quad t\geq 0,
\end{align}
for $\beta<1/2$.
They work under the restriction to solutions that do not spend time at zero, a property which is also crucial in all our arguments. Note that the index of self-similarity of \eqref{bass} is $a=\frac{1}{2-2\beta}<1$ and the H\"older continuity of the coefficient becomes worse when the self-similarity index decreases. For the classification problem of real-valued self-similar Markov processes the assumption $a=1$ could be imposed without loss of generality but it would be interesting to see if pathwise uniqueness among symmetric solutions holds for the generalized version of the SDE \eqref{Equation} that should describe all symmetric real-valued self-similar Markov processes of index $a$:
\begin{align*}
	Z_t&=z+\sigma \int_0^t |Z_s|^{1-\frac{1}{2a}}\,dB_s+ \bigg(\Psi(a)+\int_{-\infty}^0 (u-1) V(du)\bigg) \int_0^t \sign(Z_s) Z_s^{1-\frac{1}{a}}\,ds\\
	&\quad+\int_0^t \int_0^{\frac{1}{|Z_{s-}|^a}}\int_\R Z_{s-}(u-1)(\mathcal{N-N'})(ds,dr,du),
\end{align*}
with the same definitions as in Proposition \ref{PropositionMain}.
This generalization of \eqref{Equation} can be derived from the Lamperti-Kiu representation as we do in Section \ref{S1} for the special case $a=1$.
For the drift and diffusive coefficients the self-similarity index $a=1$ separates between a regime of H\"older continuity ($a>1$) and a regime with singular drift ($a<1$). Moreover, for all $a>0$ we find lack of monotonicity in the Poissonian integral and it seems that this terms forces the biggest troubles.

\subsection*{Organization of the Article}	 
	In Section \ref{S1} we prove the jump-type SDE reformulation of the Lamperti-Kiu representation for real-valued self-similar Markov processes. The proofs are given for the more general setup without symmetry and without the Assumption \textbf{(A)}. The construction of solutions to \eqref{Equation} that leave zero is presented in Section \ref{SectionConstruction}. Self-similarity and the strong Markov property are deduced from moment equations which imply uniqueness of one-dimensional marginals for solutions to the SDE \eqref{Equation}. Finally, the link to the classification problem of self-similar Markov processes is given in Section \ref{ProofEnde}
	
\section{Proofs}

\subsection{Lamperti-Kiu Representation via Jump-Type SDEs}\label{S1}
	We now state and prove a jump-type SDE formulation of the Lamperti-Kiu representation in the general case.
	\begin{proposition}\label{PropositionMain2}
	\textbf{(I)} There is a bijection between $\R_*$-valued self-similar Markov processes and two quintuples $(a_+,{\sigma^2_+},\Pi_+,q_+,V_+)$ and $(a_-,{\sigma^2_-},\Pi_-,q_-,V_-)$ consisting of
	\begin{itemize}
		\item two triplets $(a_{\pm},{\sigma^2_{\pm}},\Pi_{\pm})$ of L\'evy processes  killed at rates $q_{\pm}$ with Laplace exponents $\Psi_{\pm}$,
		\item two finite measures $V_{\pm}(du)$ on $(-\infty,0)$.
	\end{itemize}
	\smallskip
	\textbf{(II)}  Given two quintuples $(a_{\pm},{\sigma^2_{\pm}},\Pi_{\pm},q_\pm,V_\pm)$, a real-valued self-similar Markov processes issued from $z\in\R_*$ can be constructed as strong solution to
	\begin{align*}
	\begin{split}
		Z_t&=z+\Bigg[\bigg(a_++\frac{\sigma_+^2}{2}+\int_{|u|\leq 1}\big(e^{u}-1-u)\,\Pi_+(du)\bigg)\int_0^t \1_{\{Z_{s}>0\}}\,ds+\sigma_+\int_0^t \sqrt{|Z_s|} \1_{\{Z_{s}>0\}}\,dB_+(s)\\
		&\quad+\int_0^t\int_0^{\frac{1}{|Z_{s-}|}}  \int_{\R\backslash (1/e,e)}\1_{\{Z_{s-}>0\}}Z_{s-}\big({u}-1\big)\,\mathcal N_+(ds,dr,du)\\
		&\quad+\int_0^t\int_0^{\frac{1}{|Z_{s-}|}}  \int_{1/e}^e\1_{\{Z_{s-}>0\}}Z_{s-}\big({u}-1\big)\,(\mathcal N_+-{\mathcal N_+}')(ds,dr,du)\Bigg]\\
		&\quad+\Bigg[\bigg(a_-+\frac{\sigma_-^2}{2}+\int_{|u|\leq 1}\big(e^{u}-1-u)\,\Pi_-(du)\bigg)\int_0^t \1_{\{Z_{s}<0\}}\,ds+\sigma_-\int_0^t \sqrt{|Z_s|} \1_{\{Z_{s}<0\}}\,dB_-(s)\\
		& \quad+\int_0^t\int_0^{\frac{1}{|Z_{s-}|}}  \int_{\R\backslash (1/e,e)}\1_{\{Z_{s-}<0\}}Z_{s-}\big({u}-1\big)\,\mathcal N_-(ds,dr,du)\\
		&\quad+\int_0^t\int_0^{\frac{1}{|Z_{s-}|}}  \int_{1/e}^e\1_{\{Z_{s-}<0\}}Z_{s-}\big({u}-1\big)\,(\mathcal N_--{\mathcal N_-}')(ds,dr,du)\Bigg],\quad t\leq T_0.
	\end{split}
		\end{align*}
		Here, $B_{\pm}$ are standard Brownian motions and $\mathcal N_{\pm}$ are independent Poisson point processes on $(0,\infty)\times (0,\infty)\times (-\infty,\infty)$ with intensity measure $\mathcal N_{\pm}'(ds,dr,du)=ds\otimes dr\otimes \bar\Pi_{\pm}(du)$ according to the piecewise 			definition
		\begin{itemize}
			\item ${\bar\Pi^\pm}_{\big|(0,\infty)}(du)$ are the image measures under $\R\ni u\mapsto e^u$ of $\Pi_\pm$,
			\item $\bar\Pi^\pm(\{0\})=q_\pm$,
			\item $\bar\Pi^\pm_{\big|(-\infty,0)}(du)=V_\pm(du)$.
		\end{itemize}
\end{proposition}
	Before proving the proposition, let us quickly consider part {(II)} for two special cases.
	\begin{example}\label{e1}
	With the choice $z>0$ and
	\begin{align}\label{LLL}
	\begin{split}
		(a_+,{\sigma^2_{+}},\Pi_+,q_+,V_+)&=(a,\sigma^2,\Pi,q,0)\\
		(a_{-},{\sigma^2_{-}},\Pi_-,q_-,V_-)&=(0,0,0,0,0)
	\end{split}
	\end{align}
	zero is not crossed and, dropping the subscripts, the SDE simplifies to 
		\begin{align*}
	\begin{split}
		Z_t&=z +\bigg(a+\frac{\sigma^2}{2}+\int_{|u|\leq 1}\big(e^{u}-1-u)\,\Pi(du)\bigg)t+\sigma\int_0^t \sqrt{Z_s}\,dB(s)\\
		&\quad+\int_0^t\int_0^{\frac{1}{Z_{s-}}}  \int_{\R_+\backslash (e^{-1},e)}Z_{s-}\big({u}-1\big)\,\mathcal N(ds,dr,du)\\
		&\quad+\int_0^t\int_0^{\frac{1}{Z_{s-}}}  \int_{e^{-1}}^eZ_{s-}\big({u}-1\big)\,(\mathcal N-{\mathcal N}')(ds,dr,du),\quad t\leq T_0.
	\end{split}
	\end{align*}
	Under the additional assumption $\int_{\R_+\backslash (e^{-1},e)} (u-1)\, \bar\Pi(du)=\int_{|u|>1} (e^u-1) \, \Pi(du)<\infty$ one can use the L\'evy-Khintchin representation to  simplify by adding and subtracting the finite compensator integral to get
		\begin{align}\label{SDELeifMatyas}
		\begin{split}
			Z_t&=z+\Psi(1)\,t+\sigma\int_0^t \sqrt{Z_s}\,dB_s\\
			 &\quad+\int_0^t\int_0^{\frac{1}{Z_{s-}}}  \int_0^\infty Z_{s-}\big({u}-1\big)\,(\mathcal N-{\mathcal N}')(ds,dr,du),\quad t\leq T_0.
			\end{split}
		\end{align}
	The SDE \eqref{SDELeifMatyas} was already derived in \cite{DB} as reformulation of Lamperti's transformation for positive self-similar Markov processes;  strong existence and pathwise uniqueness for $t\geq 0$ was proved in \cite{DB} and also by Li and Pu \cite{LP}.
	\end{example}
	The next example shows that Proposition \ref{PropositionMain} is a special case of Proposition \ref{PropositionMain2} noting that \eqref{asss} holds trivially if $\bar\Pi$ is concentrated on $[-1,1]$.
\begin{example}
	Let us assume the symmetry
\begin{align*}
	(a_+,{\sigma^2_{+}},\Pi_+,q_+,V_+)&=(a_{-},{\sigma^2_{-}},\Pi_-,q_-,V_-)
\end{align*}
	and 
	\begin{align}\label{asss}
		\int_{|u|>1} (e^u-1) \, \Pi_\pm(du)+\int_{-\infty}^0 u\,V_\pm(du)<\infty.
	\end{align}
	The noises $B_\pm$, $\mathcal N_\pm$ can be replaced in this special case by $B,\mathcal N$ due to the symmetry assumption and  the independence of increments. Adding and subtracting the compensator integral as in Example \ref{e1} yields the SDE \eqref{Equation}.
\end{example}	
	
\begin{proof}[Proof of Proposition \ref{PropositionMain2}]
	To safe notation, let us assume throughout the proof $z>0$; for $z<0$ the arguments follow the same lines interchanging odd and even.\\
	We start with a reminder of the main result of \cite{CPR}: the Lamperti-Kiu representation formulated as time-changed exponential of a complex-valued L\'evy process.	
	Let $\xi^\pm$ be real-valued L\'evy processes with triplets $(a_\pm,\sigma^2_\pm,\Pi_\pm)$ killed at rates $q_\pm$ (formalized here as jump to $-\infty$ without causing technical complication since the process will be absorbed at the first occurrence) as in the formulation of the proposition, $\zeta^\pm$ exponential random variables with parameters $p_\pm=V_\pm(\R)$. If we denote (without confusion) by $V_\pm$ equally negative random variables with probability distribution $V_\pm(du)/V_\pm(\R)$, then $U^\pm:=\log(|V_\pm|)$ are real-valued random-variables (for the trivial case $V_\pm(\R)=0$ we define $U^\pm=0$). Further, suppose that $\xi^\pm, \zeta^\pm, U^\pm$ are independent. We consider the sequence $\big((\xi^{k},\zeta^{k}, U^{k}), k\geq  0\big)$ given by
	\begin{align}\label{cd}
		(\xi^{k},\zeta^{k}, U^{k})=\begin{cases}
			(\xi^{+,k},\zeta^{+,k}, U^{+,k}): \quad k\text{ even (including }k=0),\\
			(\xi^{-,k},\zeta^{-,k}, U^{-,k}): \quad k\text{ odd},
		\end{cases}
	\end{align}
	where $(\xi^{\pm,k},\zeta^{\pm,k}, U^{\pm,k})\stackrel{\mathcal L}{=}(\xi^\pm, \zeta^\pm,U^\pm)$ are independent. Let $(T_k,k\geq 0)$ be the sequence defined by 
	\begin{align*}
		T_0=0,\quad T_n=\sum_{k=0}^{n-1} \zeta^{k}, \quad n \geq 1,
	\end{align*}
	and $(N_t,t\geq 0)$ the alternating renewal type process
	\begin{align*}
		N_t=\max \big\{n\geq 0: T_n\leq t\big\}.
	\end{align*}
	For simplicity, the abbreviation
	\begin{align*}
		\sigma_t&=t-T_{N_t},\qquad\quad \xi_{\sigma_t}=\xi^{N_t}_{\sigma_t},\qquad\quad \xi_{\zeta}^{k}=\xi_{\zeta^{k}}^{k},
	\end{align*}
	is used. With the notation, the Lamperti-Kiu representation becomes
	\begin{align}\label{Z}
		Z_t=z\exp\big(\mathcal E_{\tau(t|z|^{-1})}\big),\quad t\leq T_0,
	\end{align}
	where
	\begin{align*}
		\mathcal E_t=\xi_{\sigma_t}+\sum_{k=1}^{N_t-1}\Big(\xi_{\zeta}^{k}+U^{k}\Big)+i\pi N_t
	\end{align*}
	and 
	\begin{align*}
		\tau(t):=\inf\bigg\{s\geq 0:\int_0^s\big| \exp\left(\mathcal E_r\right) \big| dr>t\bigg\}.
	\end{align*}
	Theorem 6 of \cite{CPR} states that \eqref{Z} defines an $\R_*$-valued self-similar Markov process issued from $z$ and conversely every $\R_*$-valued self-similar Markov process can be represented via \eqref{Z} with two quintuples as in the statement of part {(I)} of the proposition. Recall that throughout this article we suppose the index of self-similarity is $1$.
	\smallskip
	
	Note that, if $U^\pm=0$ then $N_t=0, \sigma_t=t$ and \eqref{Z} simplifies to Lamperti's representation \eqref{LT}.	
	\smallskip	
	
	The rest of the proof is concerned with part {(II)}, the reformulation of \eqref{Z} via jump-type SDEs for which the L\'evy-It\=o representation is applied to the occurring L\'evy process $\xi^{\pm, k}$. Since the sequence of L\'evy processes is independent and runs on disjoint time-intervals the same driving noises can be used for all $k\geq 1$.	
	The occurring Brownian motions are denoted by $W^1_\pm$ and the Poisson point processes with intensities $ds\otimes \Pi_\pm(du)$ by $\mathcal N^1_\pm$. Recall that the killing is included by atoms at $-\infty$ with weights $q_\pm$.	
	 Additionally, since the jumps in imaginary direction and $U^{\pm,k}$ come at same times, they can be added according to Poisson point processes $\mathcal M^1_\pm$ on $(0,\infty)\times \R$ with intensity measures ${\mathcal M^1_\pm}'(ds,dv)=p_\pm ds\otimes U^\pm(du)$, where $U^\pm(du)$ denotes the probability law of $U^\pm$. Since $U^\pm(du)$ is a probability measure the jump-rate of $\mathcal M^i_\pm$ is $p_\pm$. The L\'evy-It\=o type representation for $\mathcal E$ can now be written as
	\begin{align*}
		 \mathcal E_t
		&=\Bigg[a_+\int_0^t \1_{\{Im(\mathcal E_s)\text{ even}\}}\,ds+\sigma_+\int_0^t \1_{\{Im(\mathcal E_s)\text{ even}\}}\,dW^1_+(s)\\
		&\quad+\int_0^t \int_{|u|\leq 1} \1_{\{Im(\mathcal E_{s-})\text{ even}\}}u\,(\mathcal N^1_+-{\mathcal N^1_+}')(ds,du)
				+\int_0^t \int_{|u|> 1} \1_{\{Im(\mathcal E_{s-})\text{ even}\}}u\,\mathcal N^1_+(ds,du)\\
						&\quad+ \int_0^t \int_\R \1_{\{Im(\mathcal E_{s-})\text{ even}\}}(u+i\pi )\,\mathcal M^1_+(ds,du) \Bigg]\\
		&\quad+\Bigg[ a_-\int_0^t \1_{\{Im(\mathcal E_s)\text{ odd}\}}\,ds+\sigma_-\int_0^t \1_{\{Im(\mathcal E_s)\text{ odd}\}}\,dW^1_-(s)\\
				&\quad+\int_0^t \int_{|u|\leq 1} \1_{\{Im(\mathcal E_{s-})\text{ odd}\}}u\,(\mathcal N^1_--{\mathcal N^1_-}')(ds,du)
								+\int_0^t \int_{|u|> 1} \1_{\{Im(\mathcal E_{s-})\text{ odd}\}}u\,\mathcal N^1_-(ds,du)\\
				&\quad+\int_0^t \int_\R \1_{\{Im(\mathcal E_{s-})\text{ odd}\}}(u+i\pi )\,\mathcal M^1_-(ds,du)\Bigg]
	\end{align*} 
	with the convention $0$ is even.
	If now we set $\eta_t=z\exp(\mathcal E_t)$, apply It\=o's lemma and use
	\begin{align*}	
		\1_{\{Im(\mathcal E_s)\text{ even}\}}=\1_{\{\eta_s>0\}},\quad\quad\quad	\1_{\{Im(\mathcal E_s)\text{ odd}\}}=\1_{\{\eta_s<0\}},
	\end{align*}		
	then we obtain
	\begin{align*}
		&\quad \eta_t\\
		&=z+\Bigg[a_+\int_0^t \eta_s\1_{\{\eta_s>0\}}\,ds+\frac{\sigma_+^2}{2}\int_0^t  \1_{\{\eta_s>0\}}\eta_s \,ds+\sigma_+\int_0^t \eta_s \1_{\{\eta_s>0\}}\,dW^1_+(s)\\
		&\quad+\int_0^t \int_{|u|\leq 1}\1_{\{\eta_{s-}>0\}}\eta_{s-}\big(e^{u}-1\big)\,(\mathcal N^1_+-{\mathcal N^1_+}')(ds,du)
				+\int_0^t \int_{|u|> 1}\1_{\{\eta_{s-}>0\}}\eta_{s-}\big(e^{u}-1\big)\,\mathcal N^1_+(ds,du)\\
						&\quad+\int_0^t \int_{|u|\leq 1}\1_{\{\eta_s>0\}}\eta_{s-}\big(e^{u}-1-u)\,ds\,\Pi_+(du)
			+ \int_0^t \int_\R \eta_{s-} \1_{\{\eta_{s-}>0\}}\big(-e^{u}-1\big)\,\mathcal M^1_+(ds,du) \Bigg]\\
		&\quad+\Bigg[ a_-\int_0^t \eta_s\1_{\{\eta_{s}<0\}}\,ds+\frac{\sigma_-^2}{2}\int_0^t  \1_{\{\eta_s<0\}}\eta_s \,ds+\sigma_-\int_0^t \eta_s\1_{\{\eta_{s}<0\}}\,dW^1_-(s)\\
		&\quad+\int_0^t \int_{|u|\leq 1}\1_{\{\eta_{s-}<0\}}\eta_{s-}\big(e^{u}-1\big)\,(\mathcal N^1_--{\mathcal N^1_-}')(ds,du)
				+\int_0^t \int_{|u|> 1}\1_{\{\eta_{s-}<0\}}\eta_{s-}\big(e^{u}-1\big)\,\mathcal N^1_-(ds,du)\\
										&\quad+\int_0^t \int_{|u|\leq 1}\1_{\{\eta_{s-}<0\}}\eta_{s-}\big(e^{u}-1-u\big)\,ds\,\Pi_-(du)
			+ \int_0^t\int_\R \1_{\{\eta_{s-}<0\}}\eta_{s-} \big(-e^{u}-1\big)\,\mathcal M^1_-(ds,du)\Bigg].
	\end{align*}
	To incorporate the time-change $\tau$ we follow closely the arguments of Proposition 3.13 of \cite{DB} for the special case \eqref{LLL}.	
	 Since the arguments are almost identical, we refer for the verification of intermediate steps to the careful treatment in \cite{DB}.
	 \smallskip	
	
	 Let us first denote by $(t_n,\Delta_n)_{n\in\N}$ an arbitrary labeling of the pairs associated to jump times and jump sizes of $\left(\mathcal E_{\tau(tz^{-1})} \right)_{t\in[0,T_0)}$ and more precisely by  $(t_n,\Delta^\pm_n)_{n\in\N}$ the  subset of jumps due to $\mathcal N^\pm$ and by  $(t_n,\bar\Delta^\pm_n)_{n\in\N}$ the subset of jumps due to $\mathcal M^\pm$. We can assume that we are given additionally independent Wiener processes $(\bar W_\pm(t))_{t\geq 0}$,
 Poisson random measures $\mathcal {P}^1_\pm$ on $(0,\infty)\times (0,\infty)\times\R$ with intensity
 measure $\dd s\otimes \dd r\otimes\Pi(\dd u)$  and Poisson point processes $\mathcal {P}^2_\pm$ on $(0,\infty)\times (0,\infty)\times\R$ with intensity
 measure $p_\pm ds\otimes dr\otimes  U^\pm(du)$ that generate a filtration $(\mathcal H_t)_{t\geq 0}$. 
Additionally, we choose an independent sequence of random variables $(R_n)_{n\in\N}$ uniformly distributed
 on $(0,1)$ such that $R_n$ is $\cH_{t_n}$-measurable and independent of $\cH_{t_n-}$ and define 
	 \begin{align*}
		   W^2_\pm(t) &= \int_0^t \sign(Z_s)\sqrt{|Z_s|}\, dW^1_\pm (\tau(s|z|^{-1})) + \int_0^t \mathbf 1_{\{Z_s= 0\}}\,\dd \bar W_\pm(s),\\
		      \mathcal N^2_\pm(A_1\times A_2\times A_3)      & =\sum_{n=1}^\infty{\mathbf 1_{\{A_1\times A_2\times A_3\}}((t_n,R_n/|Z_{t_n-}|,\Delta^\pm_n))}  \\
         &\quad+\int_{A_1}\int_{A_2}\int_{A_3}(\mathbf 1_{\{r|Z_{s-}|>1 \}} + \mathbf 1_{\{Z_{s-}=0\}} )            \,\mathcal P^1_\pm(\dd s,\dd r,\dd u),\\
         		      \mathcal M^2_\pm(A_1\times A_2\times A_3)      & =\sum_{n=1}^\infty{\mathbf 1_{\{A_1\times A_2\times A_3\}}((t_n,R_n/|Z_{t_n-}|,\bar\Delta^\pm_n))}  \\
         &\quad+\int_{A_1}\int_{A_2}\int_{A_3}(\mathbf 1_{\{r|Z_{s-}|>1 \}} + \mathbf 1_{\{Z_{s-}=0\}} )            \,\mathcal P^2_\pm(\dd s,\dd r,\dd u),
	 \end{align*}
	  for all $A_1,A_2\in\cB((0,\infty))$ and $A_3\in\cB(\R)$. It now follows from L\'evy's characterization that the $W^2_\pm$ are $\mathcal H_t$-Brownian motions:
	   \begin{align*}
 \big \langle W^2_\pm(\cdot)\big\rangle_t
    = \int_0^t |Z_s| \mathbf 1_{\{Z_s\ne 0\}} \,\dd \tau(s|z|^{-1})
       + \int_0^t \mathbf 1_{\{ Z_s=0\}}\,\dd s
    = \int_0^t \mathbf 1_{\{Z_s\ne 0\}}\,\dd s
      + \int_0^t \mathbf 1_{\{Z_s=0\}}\,\dd s
    = t.
 \end{align*}
Furthermore, to show that the $\mathcal N^2_\pm$ are $\mathcal H_t$-Poisson point processes with intensity measures $dt\otimes dr\otimes \Pi_\pm(du)$. Applying Theorems II.1.8 and II.4.8 of \cite{JS}, we need to verify

 \begin{align}\label{help1}
  \begin{split}
    &E\left(\int_0^\infty \int_0^\infty \int_\R H(s,r,u)\,\mathcal N^2_\pm(\dd s,\dd r,\dd u)\right)
= E\left(\int_0^\infty \int_0^\infty \int_\R H(s,r,u)\,\dd s\,\dd r\,\Pi_\pm(\dd u)\right)
  \end{split}
 \end{align}
 for every non-negative predictable function $H$ on $\Omega\times (0,\infty)\times(0,\infty)\times\R$.
By the definition of $\mathcal N^2_\pm$ we can write
 \begin{align}\label{help21}
  \begin{split}
   E&\left(\int_0^\infty \int_0^\infty \int_\R H(s,r,u)\,\mathcal N^2_\pm(\dd s,\dd r,\dd u)\right)\\
    &= E\left(\sum_{n=1}^\infty
         H(t_n,R_n/|Z_{t_n-}|,\Delta_n)\right)\\
    &\quad  +  E\left(\int_0^\infty \int_0^\infty \int_\R H(s,r,u)
                   (\mathbf 1_{\{r|Z_{s-}|>1\}} + \mathbf 1_{\{Z_{s-}=0\}} )
                     \,\cP^1_\pm(\dd s,\dd r,\dd u)\right).
  \end{split}
 \end{align}
To express the first summand we apply Theorem II.1.8 of Jacod and Shiryaev \cite{JS} to the
 non-negative predictable function $$\widetilde H(s,r,u):=H(s,r/|Z_{s-}|,u)
 ,\quad s>0,r>0,u\in\R,$$
 and the Poisson random measure on $(0,\infty)\times(0,\infty)\times\R$ defined by
 \[
   \widetilde\cP^1(A_1\times A_2\times A_3)
      := \sum_{n=1}^\infty {\mathbf 1_{A_1\times A_2\times A_3}((t_n,R_n,\Delta_n))},
      \qquad A_1,A_2\in\cB((0,\infty)),\;\; A_3\in\cB(\R),
 \]
 to obtain
 \begin{align*}
	 	 E\left(\sum_{n=1}^\infty H(t_n,R_n/|Z_{t_{n-}}|,\Delta_n) \right)
		 &=\E\left(\int_0^\infty \int_0^\infty \int_\R \widetilde H(s,r,u)\,
                   \widetilde\cP^1(\dd s,\dd r,\dd u)\right)\\
	     	 &= E\left(\int_0^{T_0} \int_0^1 \int_\R H(s,r/|Z_{s-}|,u)
                       \,\dd \tau({s|z|^{-1})} \,\dd r\,\Pi(\dd u) \right),
 \end{align*}
 where the second equality holds since, by construction, the compensator measure of $\widetilde\cP^1$
 is  $$\mathbf 1_{(0,T_0)}(s)\mathbf 1_{(0,1)}(r)\dd \tau(s|z|^{-1})\,\dd r\,\Pi(\dd u).$$
Utilizing a change of variable in the second coordinate of $H$,
 we can further simplify the right-hand side to
 \begin{align*}
&E\left(\int_0^{T_0} \int_0^1 \int_\R \frac{1}{|Z_{s-}|} H(s,r/ |Z_{s-}|,u) \, \dd s\,\dd r\,\Pi(\dd u) \right)\\
&\quad  = E\left(\int_0^{T_0} \int_0^{1/|Z_{s-}|}\int_\R  H(s,r,u) \,\dd s\,\dd r\,\Pi(\dd u) \right).
 \end{align*}
Similarly, applying Theorem II.1.8 of Jacod and Shiryaev \cite{JS} to $\cP^1_\pm$, the second summand of
 {the right hand side of} \eqref{help21} equals
  \begin{align*}
   	E&\left(\int_0^\infty \int_0^\infty \int_\R
        H(s,r,u)(\mathbf 1_{\{r|Z_{s-}|>1\}} + \mathbf 1_{\{Z_{s-}=0\}} )\,\cP^1_\pm(\dd s,\dd r,\dd u)\right) \\
     &= E\left(\int_0^{T_0} \int_{1/|Z_{s}|}^\infty\int_\R  H(s,r,u) \, \dd s \,\dd r\,\Pi(\dd u) \right)
         + E\left(\int_{T_0}^\infty \int_0^\infty\int_\R  H(s,r,u) \, \dd s \,\dd r\,\Pi(\dd u) \right).
   \end{align*}
Adding {the right hand sides of the above two equalities, by \eqref{help21}, we have} \eqref{help1}.
\smallskip

Similarly, one can show that the $\mathcal M^2_\pm$ are $\mathcal H_t$-Poisson point processes with intensity measures $p_\pm dt\otimes dr\otimes U^\pm(du)$.
	  \smallskip
	  
	Plugging-in the new Brownian motion we obtain
	\begin{align*}
		\sigma_\pm\int_0^{\tau(t|z|^{-1})}\1_{\{\eta_s>0\}}\eta_s \, dW^2_\pm(s)&=\sigma_\pm\int_0^t  \1_{\{Z_s>0\}}Z_s  \,dW^2_\pm(\tau(sz^{-1})) \\
		&=\sigma_\pm\int_0^t \1_{\{Z_s>0\}} \sqrt{|Z_s|}\sign(Z_s)\sqrt{|Z_s|}  \,dW^2_\pm({\tau({sz^{-1})}})\\
		&=\sigma_\pm\int_0^t \1_{\{Z_s>0\}}\sqrt{ |Z_s|}\,dW^1_\pm(s) ,  \qquad t\leq T_0,
	\end{align*}
	and analogously for the negative part. Comparing one-by-one the jumps of the Poisson point processes we also find, by construction of the new point measures,
	\begin{align*}		
		&\quad \int_0^{\tau(t|z|^{-1})}  \int_{|u|> 1}\1_{\{\eta_{s-}>0\}}\eta_{s-}\big(e^{u}-1\big)\,\mathcal N^1_+(ds,du)\\
		&=\int_0^t\int_0^{\frac{1}{|Z_{s-}|}}  \int_{|u|> 1}\1_{\{Z_{s-}>0\}}Z_{s-}\big(e^{u}-1\big)\,\mathcal N^2_+(ds,dr,du),\quad t\leq T_0,
	\end{align*}	
	and
	\begin{align*}
		&\quad \int_0^{\tau(t|z|^{-1})}  \int_{|u|\leq  1}\1_{\{\eta_{s-}>0\}}\eta_{s-}\big(e^{u}-1\big)\,(\mathcal N^1_+-{\mathcal N^1_+}')(ds,du)\\
		&=\int_0^t\int_0^{\frac{1}{|Z_{s-}|}}  \int_{|u|\leq 1}\1_{\{Z_{s-}>0\}}Z_{s-}\big(e^{u}-1\big)\,(\mathcal N^2_+-{\mathcal N^2_+}')(ds,dr,du),\quad t\leq T_0,
	\end{align*}
	and
	\begin{align*}
		 &\quad\int_0^{\tau(t|z|^{-1})}\int_\R \1_{\{\eta_{s-}>0\}}\eta_{s-} \big(-e^{u}-1\big)\,\mathcal M^1_+(ds,du)\\
		 &= \int_0^t\int_0^{\frac{1}{|Z_{s-}|}} \int_\R \1_{\{Z_{s-}>0\}}Z_{s-} \big(-e^{u}-1\big)\,\mathcal M^2_+(ds,dr,du),\quad t\leq T_0,
	\end{align*}
	and analogously for the negative parts. Finally, ordinary change of time yields
	\begin{align*}
		\frac{\sigma_\pm^2}{2}	\int_0^{\tau(t|z|^{-1})} \eta_s \1_{\{\eta_s>0\}} \,ds&=\frac{\sigma_\pm^2}{2}\int_0^t \1_{\{Z_s>0\}} \,ds,\quad t\leq T_0.
	\end{align*}

	Plugging-into the integral equation derived for $\eta$, we find that $Z$ satisfies 
	\begin{align*}
		Z_t&=z+\Bigg[\bigg(a_++\frac{\sigma_+^2}{2}+\int_{|u|\leq 1}\big(e^{u}-1-u)\,\Pi_+(du)\bigg)\int_0^t \1_{\{Z_{s}>0\}}\,ds+\sigma_+\int_0^t \sqrt{|Z_s|} \1_{\{Z_{s}>0\}}\,dW^2_+(s)\\
		&\quad+\int_0^t\int_0^{\frac{1}{|Z_{s-}|}}  \int_{|u|>1}\1_{\{Z_{s-}>0\}}Z_{s-}\big(e^{u}-1\big)\,\mathcal N^2_+(ds,dr,du)\\
		&\quad+\int_0^t\int_0^{\frac{1}{|Z_{s-}|}}  \int_{|u|\leq 1}\1_{\{Z_{s-}>0\}}Z_{s-}\big(e^{u}-1\big)\,(\mathcal N^2_+-{\mathcal N^2_+}')(ds,dr,du)\\
		&\quad + \int_0^t\int_0^{\frac{1}{|Z_{s-}|}} \int_\R \1_{\{Z_{s-}>0\}}Z_{s-} \big(-e^{v}-1\big)\,\mathcal M^2_+(ds,dr,dv) \Bigg]\\
		&\quad+\Bigg[\bigg(a_-+\frac{\sigma_-^2}{2}+\int_{|u|\leq 1}\big(e^{u}-1-u)\,\Pi_-(du)\bigg)\int_0^t \1_{\{Z_{s}<0\}}\,ds+\sigma_-\int_0^t \sqrt{|Z_s|} \1_{\{Z_{s}<0\}}\,dW^2_-(s)\\
		& \quad+\int_0^t\int_0^{\frac{1}{|Z_{s-}|}}  \int_{|u|>1}\1_{\{Z_{s-}<0\}}Z_{s-}\big(e^{u}-1\big)\,\mathcal N^2_-(ds,dr,du)\\
		&\quad+\int_0^t\int_0^{\frac{1}{|Z_{s-}|}}  \int_{|u|\leq 1}\1_{\{Z_{s-}<0\}}Z_{s-}\big(e^{u}-1\big)\,(\mathcal N^2_--{\mathcal N^2_-}')(ds,dr,du)\\
		&\quad + \int_0^t\int_0^{\frac{1}{|Z_{s-}|}} \int_\R \1_{\{Z_{s-}<0\}}Z_{s-} \big(-e^{v}-1\big)\,\mathcal M^2_-(ds,dr,dv) \Bigg],\quad t\leq T_0.
	\end{align*}	
	The final step is only for notational convenience: We change the coordinates for the jumps of $\mathcal N^2_\pm$, $\mathcal M^2_\pm$ in order to combine the integrals to integrals driven by $\mathcal N_\pm$ as in the statement of the Proposition.	
	\end{proof}
		

\subsection{Construction of Real-Valued Self-Similar Processes}\label{SectionConstruction}
	The aim of this section is to construct real-valued self-similar Markov processes that leave zero continuously with Lamperti-Kiu quintuple $(a,\sigma^2,\Pi,q,V)$ whenever Condition \eqref{Condition} is valid. We construct a symmetric approximating sequence for the martingale problem corresponding to the SDE \eqref{Equation} and use moment equations of Bertoin and Yor \cite{BY2} to show that limit points are Markovian and self-similar.
	\smallskip
	
	 Recall that for a generator $\mathcal A$ defined on a suitably chosen subset $\mathcal D(\mathcal A)$ of the bounded and measurable functions $B(\R)$ mapping $\R$ into $\R$ a stochastic process $Z$ is said to be a solution to the martingale problem $(\mathcal A, \nu)$ corresponding to $\mathcal A$ with initial distribution $\nu$ if for all $f \in \mathcal D(\mathcal A)$
	\begin{align*}
		M_t^f=f(Z_t)-\int_0^t \mathcal A f(Z_s)\,ds,\quad t\geq 0,
	\end{align*}
	is a martingale and $Z_0$ is distributed according to $\nu$. The next proposition is standard; it is included for completeness and since the used estimates will appear several times in the sequel.
	\begin{proposition}\label{Proposition}
		A stochastic process $Z$ is a weak solution to the SDE \eqref{Equation} issued from $z\in\R$ if and only if it satisfies the martingale problem  $(\mathcal A,\delta_z)$ corresponding to the generator 
 		\begin{align}\label{Generator}
			\begin{split}
	  			(\cA f)(z)
  				&:=\bigg(\Psi(1)+\int_{-1}^0 (u-1) V(du)\bigg)\sign(z)f'(z)+\frac{\sigma^2}{2}|z|f''(z)\\
      	   		         &\quad + \int_0^{\infty} \int_{{[-1,1]}}\mathbf 1_{\{r|z|\leq 1\}}\Big(f\big(u z\big)-f(z) -f'(z)z(u-1)\Big) \,\dd r\,\bar\Pi(\dd u),\quad z\in\R,
			\end{split}
 		\end{align}
		acting on the infinitely differentiable functions with compact support $C_c^\infty(\R)$.		
	\end{proposition}
	\begin{proof}
	 	Let us first suppose $Z$ is a weak solution to the SDE \eqref{Equation}. Applying It\=o's formula with $f\in C^\infty_c(\R)$ yields
		\begin{align*}
			M_t^f&=f(Z_t)-f(z)-\bigg(\Psi(1)+\int_{-1}^0 (u-1) V(du)\bigg)\int_0^t f'(Z_s)\sign{(Z_s)}\,ds-\frac{\sigma^2}{2} \int_0^t f''(Z_s)|Z_s|\,ds\\
			&\quad-\int_0^t \int_0^{\infty }\int_{-1}^1\big(f(Z_s +\1_{\{r|Z_s|\leq 1\}}Z_s(u-1))-f(Z_s)-f'(Z_s)\1_{\{r|Z_s|\leq 1\}}Z_{s}(u-1)\big)\,ds\,dr\,\bar\Pi(du)\\
			&= f(Z_t)-f(z)-\int_0^t \mathcal A f(Z_s)\,ds
		\end{align*}
		is a local martingale, where 
		\begin{align*}
			M_t^f
			&=\sigma \int_0^t f'(Z_s)\sqrt{|Z_s|} \,dB_s\\
			&\quad+\int_0^t \int_0^\infty\int_{-1}^1\1_{\{r|Z_{s-}|\leq 1\}}\big(f(Z_{s-}+Z_{s-}(u-1))-f(Z_{s-})\big)(\mathcal{N-N'})(ds,dr,du).
		\end{align*}
		Moreover, it is easy to see that $M^f$ is a true martingale. Indeed, by Theorem 51 of \cite{Protter} it suffices to verify $\E[\sup_{t\leq T}|M^f_t|]<\infty$ for all $T>0$. Applying the Burkholder-Davis-Gundy inequaliy (for the non-continuous martingale see \cite{DM}, p. 287) and the 		simple estimate 
		$\E\big[\sup_{t\leq T}|M^f_t|\big]\leq 1+\E\big[\sup_{t\leq T}|M^f_t|^2\big]$, we obtain
		\begin{align*}
			\E\big[\sup_{t\leq T}|M^f_t|\big]
			&\leq 1+2\sigma^2 \E\bigg[\sup_{t\leq T}\bigg|\int_0^t f'(Z_s)\sqrt{|Z_s|}\,dB_s\bigg|^2\bigg]\\
			&\quad+2\E\bigg[\sup_{t\leq T}\bigg|\int_0^t\int_0^{\frac{1}{|Z_{s-}|}}\int_{[-1,1]} \Big(f(Z_{s-}+Z_{s-}(u-1))-f(Z_{s-})\Big)\mathcal{(N-N')}(ds,dr,du)\bigg|^2\bigg]\\
			&\leq 1+ C\E\bigg[\int_0^T f'(Z_s)^2|Z_s|\,ds\bigg]\\
			&\quad+C\E\bigg[\int_0^T\int_0^{\frac{1}{|Z_{s-}|}}\int_{[-1,1]} \Big(f(Z_s+Z_s(u-1))-f(Z_s)\Big)^2ds\,dr\,\bar\Pi(du)\bigg].
		\end{align*}
		By Taylor's formula and the boundedness of $f'$ we find the upper bound
		\begin{align}\label{l'}
			\E\big[\sup_{t\leq T}|M^f_t|\big]& \leq 1+\big(\sup_z f'(z)\big)^2\bigg(C+C\int_{[-1,1]} (u-1)^2\,\bar\Pi(du)\bigg)\E\bigg[\int_0^T |Z_s|\,ds\bigg].
		\end{align}
		Note that the definition of $\bar\Pi$ implies 
		\begin{align*}
			\int_{-1}^1 (u-1)^2\,\bar\Pi(du)&=\int_{-1}^0 (u-1)^2\,V(du)+\int_{-\infty}^0 \big(e^u-1\big)^2\,\Pi(du)\\
			&\leq 2V\big([-1,0)\big)+\Pi\big((-\infty,-1]\big)+C \int_{-1}^0 u^2 \,\Pi(du)
		\end{align*}
		which is finite since $V(du)$ is a finite measure and $\Pi$ is a L\'evy measure.\\
		Next, we show that $\E[\int_0^T |Z_s|\,ds]$ is finite. From It\=o's isometry, Taylor's theorem and the estimate $(a_1+...+a_n)^2\leq n(a_1^2+...+a_n^2)$ for $a_i\in\R$ we get
		\begin{align*}
			&\quad\E\big[|Z_t|\big]\\
			&\leq 1+\E\big[|Z_t|^2\big]\\
			&\leq 1+ 4 z^2+ 4\bigg(\Psi(1)+\int_{-1}^0 (u-1) V(du)\bigg)^2\,t^2+4\sigma^2\E\bigg[\int_0^t |Z_s|\,ds\bigg]\\
			&\quad+4\E\bigg[\int_0^t \int_0^{\frac{1}{|Z_s|}}\int_{-1}^1 Z_s^2(u-1)^2\,ds\,dr\,\bar\Pi(du)\bigg]\\
			&\leq 1+4z^2+4\bigg(\Psi(1)+\int_{-1}^0 (u-1) V(du)\bigg)^2\,t^2+4\bigg(\sigma^2
			 +\int_{-1}^1 (u-1)^2\,\bar\Pi(du) \bigg) \int_0^T \E\big[|Z_s|\big] \,ds.
		\end{align*}
		Hence, Gronwall's inequality implies that $\E\big[|Z_t|\big]$ grows at most exponentially so that $\E\big[\int_0^T |Z_s|\,ds\big]$ is finite by Fubini's theorem. Now we can deduce from \eqref{l'} that $M_t^f$ is a martingale and the first part of the proof is complete.\\
		
		Conversely, suppose the law of the process $Z$ is a solution to the martingale problem $(\mathcal A,\delta_z)$. By a standard stopping time argument to allow for the test-function $f(z)=z$, we have
		\begin{align*}
			Z_t=z+\bigg(\Psi(1)+\int_{-1}^0 (u-1) V(du)\bigg) \int_0^t \sign(Z_s)\,ds+M_t,\quad t\geq 0,
		\end{align*}
		almost surely, for a  square-integrable martingale $M$ that we have to identify. Let $\mathcal C(ds,dz)$ be the optional random measure on $[0,\infty)\times \R$ defined by the jumps of $Z$:
		\begin{align*}
			\mathcal C(ds,dz)=\sum_{s>0} \1_{\{\Delta Z_s\neq 0\}}\delta_{(s, \Delta Z_{s})}(ds,dz),
		\end{align*}
		where $\Delta Z_s=Z_s-Z_{s-}$ is the jump of $Z$ at time $s$. If ${\mathcal C'}$ denotes the predictable compensator of $\mathcal C$, then page 376 of \cite{DM} shows that 
		\begin{align}\label{SM}
			Z_t=z+\bigg(\Psi(1)+\int_{-1}^0 (u-1) V(du)\bigg) \int_0^t \sign(Z_s)\,ds+M^c_t+M^d_t
		\end{align}
		for a continuous martingale $M^c$ and
		\begin{align*}
			M^d_t=\int_0^t \int_\R z \,\mathcal{(C-C')}(ds,dz).
		\end{align*}		
	 	We now have to identify the martingales $M^c$ and $M^d$. Applying It\=o's formula to the semimartingale representation \eqref{SM} of $Z$ yields
		\begin{align*}
			f(Z_t)&=f(z)+\bigg(\Psi(1)+\int_{-1}^0 (u-1) V(du)\bigg) \int_0^t f'(Z_s) \sign(Z_s)\,ds+\frac{1}{2}\int_0^t f''(Z_s)\,d[M^c_s,M^c_s]\\
			&\quad+\int_0^t \big(f(Z_{s}+z)-f(Z_{s})-f'(Z_{s})z )\big)\, \mathcal C'(ds,dz)+local\,\,martingale
		\end{align*}		
		for all $f\in C^\infty_c(\R)$. We can assume without loss of generality that the local martingale is a martingale since otherwise the rest of the proof can be carried out via localization. Comparing with the martingale problem $(\mathcal A,\delta_z)$ from \eqref{Generator} and using the uniqueness of the canonical decomposition for a semimartingale, we see that $d[M^c_s,M^c_s]= \sigma^2 |Z_s|ds$ and
		\begin{align*}
			\int_0^t\int_\R F(s,z)\, \mathcal C'(ds,dz)= \int_0^t \int_0^\infty \int_{-1}^1 F(s,\mathbf 1_{\{r|Z_s|\leq 1\}}Z_s(u-1))\,ds\,dr\,\bar\Pi(du).
		\end{align*}
		for any non-negative Borel function $F$ on $[0,\infty)\times \R$. Then we can find a Brownian motion $B$ and an independent Poisson point process $\mathcal N$ on  $(\Omega,\cG,(\cG_t)_{t\geq 0},P)$ by applying martingale representation 			theorems to \eqref{SM} (see for instance \cite{IW} page 84 and page 93).
	\end{proof}
	
		
	The construction of a solution to the martingale problem $(\mathcal A,\delta_z)$ is achieved with a series of lemmas. To give a rough idea how to construct solutions let us reconsider the simplest special case
	\begin{align}\label{bli}
			dZ_t=\sign(Z_t)\,dt+2 \sqrt{|Z_t|}\,dB_t,\quad Z_0=0,
	\end{align}
	and it's positive analogue
	\begin{align}\label{Z2}
		dZ_t&=\,dt+2\sqrt{Z_t}\,dB_t,\quad Z_0=0,
	\end{align}
	obtained for the absolute value. If $W$ is a Brownian motion, then we already noted that $Z^{(1)}_t=\sign(W_t) W_t^2$ is a weak solution to the SDE \eqref{bli} and furthermore $Z^{(2)}_t=W_t^2$ is a weak solution to the SDE \eqref{Z2}. Of course, $Z^{(1)}$ and $Z^{(2)}$ have a straight forward connection: given $Z^{(2)}$, $Z^{(1)}$ is obtained by reflecting every excursion at the origin with probability $1/2$.
	\smallskip	
	
	An analogous procedure could be applied to construct symmetric solutions for the jump-type SDE \eqref{Equation} since excursion theory for $|Z|$ exists ($|Z|$ is a positive self-similar Markov process). In the general case of the SDE \eqref{Equation} solutions additionally jump over 		zero so that a direct modification of the reflection idea seems not to work.
	In what follows we give a stochastic calculus construction that mimics the reflection idea but is robust enough to encounter jumps that change signs.\\
	
	The notations of Theorem \ref{TheoremMain} and Proposition \ref{Proposition} will be used in the sequel without explicit repetitions. 
	\begin{lemma}\label{Lemma1}
		Suppose $m\in \N$ and that $\mathcal M^m$ is a Poisson point process on $(0,\infty)\times \{-\frac 1 m, \frac 1 m \}$ independent of $B$ and $\mathcal N$ with intensity measure ${\mathcal{M}^m}'(ds,dv)=ds\otimes \Sigma(dv)$, where 
		$\Sigma(\big\{\frac 1 m \big\})=\Sigma(\big\{-\frac 1 m\big\}) =\frac m 2.$
		If we define
		\begin{align*}
			\sign_{(0)}(x)=\1_{\{x>0\}}-\1_{\{x<0\}},
		\end{align*}
		then
	there are unique strong solutions $Z^{m}$ to the SDE
	\begin{align}\label{EquationApprox}
	\begin{split}
		Z_t&=z+\bigg(\Psi(1)+\int_{-1}^0 (u-1) V(du)\bigg)\int_0^t\sign_{(0)}(Z_s)\,ds+\sigma\int_0^t \sqrt{|Z_s|}\,dB_s\\
		&\quad+\bigg(\Psi(1)+\int_{-1}^0 (|u|-1) V(du)\bigg)\int_0^t\int_{ \{\pm \frac 1 m \}}\1_{\{Z_{s-}= 0\}}v\,\mathcal M^m(ds,dv)\\
		&\quad+\int_0^t \int_0^{\frac{1}{|Z_{s-}|}\wedge m}\int_{[-1,1-\frac{1}{m}]} Z_{s-}(u-1)(\mathcal{N-N'})(ds,dr,du),\quad t\geq 0.
	\end{split}
	\end{align}
	\end{lemma}

	\begin{proof}[Proof of Lemma \ref{Lemma1}]
		Suppressing the jumps according to the point process $\mathcal N$ and integrating out $dr$ in the remaining compensator integral yields the SDE
		\begin{align*}
			\begin{split}
				Z_t&=z+\bigg(\Psi(1)+\int_{-1}^0 (u-1) V(du)\bigg)\int_0^t\sign_{(0)}(Z_s)\,ds+\sigma\int_0^t \sqrt{|Z_s|}\,dB_s\\
				&\quad+\bigg(\Psi(1)+\int_{-1}^0 (|u|-1) V(du)\bigg)\int_0^t\int_{ \{\pm \frac 1 m \}}\1_{\{Z_{s-}= 0\}}v\,\mathcal M^m(ds,dv)\\
				&\quad-\int_{[-1,1-\frac{1}{m}]} (u-1)\bar\Pi(du) \int_0^t\sign(Z_s)\big(1\wedge m|Z_s|\big)\,ds,
			\end{split}
		\end{align*}
		where we used $\frac{x}{|x|}=\sign(x)$ for $ x\neq 0$. Whenever a solution is bounded away from zero, pathwise uniqueness and strong existence holds due to the local Lipschitz property of the integrands away from zero. When the solution hits zero it remains until it jumps according to a jump of $\mathcal M^m$. Since both jump integrals only jump with bounded rate, a strong solution of \eqref{EquationApprox} can be constructed piecewise via the interlacing method.\\
		This is a standard argument, so we omit the details and refer for instance to the proof of Proposition 3.5 of \cite{DB}.
\end{proof}
	The choice of the SDE \eqref{EquationApprox} is motivated by the reflected  excursion idea for a construction of solutions: whenever solutions are away from zero, they follow the original SDE \eqref{Equation}, the "pseudo excursions" taking values in $\R_*$. At zero the pseudo excursions stop and after an exponential time a new pseudo excursion is started at a small initial state chosen symmetrically by $\mathcal M^m$. The symmetric restarting is needed to construct a symmetric process; non-symmetric restarting might be used to construct skew-self-similar Markov processes.\\
	It is crucial to redfine the sign-function to be zero at zero since otherwise the constructed process is not a solution to \eqref{EquationApprox}. As $m$ increases, the times between pseudo excursions and the new initial states tend to zero so that possible limiting processes leave zero continuously.
\smallskip
	
	 The construction shows that, for all $z$ and $m$,
\begin{align*}
	\int_0^\infty \1_{\{Z^{m}_s=0\}}\,ds=\infty,\quad a.s.,
\end{align*}
if the pseudo excursions hit zero in finite time. Hence, a priori it is possible that any limiting process $Z$ is trapped at zero. To guarantee that $Z$ is not such a trivial solution, under Condition \eqref{Condition} we are going to deduce
\begin{align}\label{abs0}
	\int_0^\infty \1_{\{Z_s=0\}}\,ds=0,\quad a.s.
\end{align}
In order to be able to verify \eqref{abs0}, the constant in front of the stochastic integral with respect to $\mathcal M^m$ turns out to be crucial. This might be surprising since in the limit $m\to\infty$ this integral vanishes without leaving a compensator term since it is a martingale. In order to show \eqref{abs0}, we show in Lemma \ref{Lemma3} below that the limiting absolute value $|Z|$ solves an SDE with constant drift. It is not clear a priori that the absolute value has constant drift since in \eqref{EquationApprox} the drift is zero at zero. To ensure that the drift for the absolute value is constant (and not zero at zero) we use that the stochastic integral with respect to $\mathcal M^m$ leaves in the absolute value the compensator integral
\begin{align*}
	&\quad\bigg(\Psi(1)+\int_{-1}^0 (|u|-1) V(du)\bigg)\int_0^t \1_{\{Z^m_s=0\}}\,ds\\
	&=\bigg(\Psi(1)+\int_{-1}^0 (u-1) V(du)\bigg)\int_0^t \1_{\{Z^m_s=0\}}\,ds+\bigg(\int_{-1}^0( |u|-u)V(du)\bigg)\int_0^t \1_{\{Z^m_s=0\}}\,ds.
\end{align*}
The summands compensate the time spend at zero that is not taken into account by $\sign_{(0)}$ for the drift and $\sign_{(0)}$ that appears as limit of $\big(\frac{1}{|Z_{s}|}\wedge m\big) Z_s$ for the compensated integral with respect to $\mathcal N$.

\begin{lemma}\label{Lemma7}
	Suppose $Z^{m}$ is as in Lemma \ref{Lemma1}, then
	\begin{align}\label{||}
		\begin{split}
			|Z^{m}_t|&=|z|+\bigg(\Psi(1)+\int_{-1}^0 (u-1) V(du)\bigg)\,t\\
			&\quad+\bigg(\int_{-1}^0 (|u|-u) V(du)\bigg) \int_0^t\big(\1_{\{Z^{m}_s=0\}}+\big(1\wedge m |Z^{m}_s|\big)\big)\,ds\\
			&\quad+\sigma \int_0^t \sign(Z^{m}_s)\sqrt{|Z^{m}_s|}\,dB_s\\
			&\quad+\bigg(\Psi(1)+\int_{-1}^0 (|u|-1) V(du)\bigg)\int_0^t\int_{ \{\pm \frac 1 m \}}\1_{\{Z^{m}_{s-}= 0\}}|v|\,\mathcal{(M}^m-{\mathcal{M}^m}')(ds,dv)\\
			&\quad+\int_0^t \int_0^{\frac{1}{|Z^{m}_{s-}|}\wedge m}\int_{[-1,1-\frac{1}{m}]}  |Z^{m}_{s-}|(|u|-1)(\mathcal{N-N'})(ds,dr,du),\quad t\geq 0,
		\end{split}
	\end{align}
	almost surely.
\end{lemma}	
\begin{proof}
Let us denote 	by $\tau_1<\tau_2<...$ the jumps of the Poissonian integral driven by $\mathcal M^m$ which are precisely the times when $Z^m$ leaves zero.  Further, $\delta_1<\delta_2<...$ denote the successive first hitting times of zero which do not accumulate since 	paths are c\`{a}dl\`{a}g and solutions only leave zero with a jump of size $\pm \frac{\Psi(1)+\int_{-1}^0 (|u|-1) V(du)}{m}$.  If we define $ZERO:= [\delta_1,\tau_1)\cup [\delta_2,\tau_2)\cup ...$, then 
	\begin{align*}	
		Z^{m}_s=0\quad \forall s \in ZERO\qquad \text{and}\qquad Z^{m}_s\neq 0\quad \forall s \notin ZERO.
	\end{align*}
	Consequently, $|Z^{m}_s|=0$ for $s \in ZERO$ so that it suffices to apply Tanaka's formula to $Z^{m}$ on $ZERO^c$. Let us first show that the semimartingale local time at zero vanishes.
	The truncation by $m$ implies that jumps are summable so that Corollary 3 on page 178 of \cite{Protter} yields
	\begin{align*}
		L_t^0&=\lim_{\eps\to 0} \frac{1}{\eps}\int_{0}^t \1_{\{|Z^{m}_s|\leq \eps\}} d[Z^{m}_s,Z^{m}_s]^c\\
				&=\lim_{\eps\to 0} \frac{\sigma^2}{\eps}\bigg[\sum_{j=1}^{i-1} \int_{\tau_j}^{\delta_{j+1}} \1_{\{|Z_s^{m}|\leq \eps\}}|Z_s^{m}|\,ds+\int_{\tau_i}^t \1_{\{|Z_s^{m}|\leq \eps\}}|Z_s^{m}|\,ds\bigg]\\
			&\leq \lim_{\eps\to 0} \sigma^2 \bigg[\sum_{j=1}^{i-1} \int_{\tau_j}^{\delta_{j+1}} \1_{\{|Z_s^{m}|\leq \eps\}}\,ds+\int_{\tau_i}^t \1_{\{|Z_s^{m}|\leq \eps\}}\,ds\bigg],\quad t\in [\tau_i,\delta_{i+1}).
	\end{align*}
	Using dominated convergence, the righthand side converges to zero since $Z^{m}$ does not spend time at zero on $ZERO^c$.
	Next, Tanaka's formula can be applied without additional local time term to deduce the semimartingale decomposition
	\begin{align*}
		&\quad |Z^{m}_t|\\
		&=|z|+\bigg(\Psi(1)+\int_{-1}^0 (u-1) V(du)\bigg)\int_0^t \sign(Z^{m}_s)\sign_{(0)}(Z^{m}_s)\,ds+\sigma \int_0^t \sign(Z^{m}_s)\sqrt{|Z^{m}_s|}\,dB_s\\
		&\quad+\int_0^t\int_{ \{\pm \frac 1 m \}}\bigg(\bigg|Z^{m}_{s-}+\bigg(\Psi(1)+\int_{-1}^0 (|u|-1) V(du)\bigg)\1_{\{Z^{m}_{s-}= 0\}}v\bigg|-\big|Z^{m}_{s-}\big|\bigg)\,\mathcal{M}^m(ds,dv)\\
				&\quad+\int_0^t \int_0^{\frac{1}{|Z^{m}_{s-}|}\wedge m}\int_{[-1,1-\frac{1}{m}]}  |Z^{m}_{s-}|(|u|-1)(\mathcal{N-N'})(ds,dr,du)\\
		&\quad+\int_0^t\big(1\wedge m |Z^{m}_s|\big)\,ds \int_{[-1,1-\frac{1}{m}]} (|u|-u)\bar\Pi(du).
	\end{align*}
	Adding and subtracting the compensator integral for $\mathcal M^m$, we obtain as a drift
	\begin{align*}
		&\quad\bigg(\Psi(1)+\int_{-1}^0 (u-1) V(du)\bigg)\int_0^t \1_{\{Z^{m}_s\neq 0\}}\,ds\\
		&\quad+\int_0^t\int_{ \{\pm \frac 1 m \}}\1_{\{Z^{m}_{s-}= 0\}}\bigg(\Psi(1)+\int_{-1}^0 (|u|-1) V(du)\bigg)v\,{\mathcal{M}^m}'(ds,dv)\\
				&\quad+\int_0^t\big(1\wedge m |Z^{m}_s|\big)\,ds \int_{[-1,1-\frac{1}{m}]} (|u|-u)\bar\Pi(du)\\
						&=\bigg(\Psi(1)+\int_{-1}^0 (u-1) V(du)\bigg)\int_0^t \1_{\{Z^{m}_s\neq 0\}}\,ds\\
						&\quad+\bigg(\bigg(\Psi(1)+\int_{-1}^0 (u-1) V(du)\bigg)+\bigg(\int_{-1}^0 (|u|-u) V(du)\bigg)\bigg)\int_0^t\int_{ \{\pm \frac 1 m \}}\1_{\{Z^{m}_s= 0\}}|v|\,\Sigma(dv)\,ds\\
				&\quad+\int_0^t\big(1\wedge m |Z^{m}_s|\big)\,ds \int_{[-1,1-\frac{1}{m}]} (|u|-u)\bar\Pi(du)\\
		&=\bigg(\Psi(1)+\int_{-1}^0 (u-1) V(du)\bigg)\,t+\bigg(\int_{-1}^0 (|u|-u) V(du)\bigg)\int_0^t \1_{\{Z^{m}_s=0\}}\,ds\\
				&\quad+\int_0^t\big(1\wedge m |Z^{m}_s|\big)\,ds \int_{[-1,1-\frac{1}{m}]} (|u|-u)\bar\Pi(du).
	\end{align*}
	Using the definition of $\bar\Pi$ we can simplify the final integral to
	\begin{align*}
		\int_{[-1,1-\frac{1}{m}]} (|u|-u)\bar\Pi(du)= \int_{-1}^0 (|u|-u)\bar\Pi(du)= \int_{-1}^0 (|u|-u)V(du)
	\end{align*}	
	from which the claim follows.
\end{proof}

Next, we show that there are limits of the sequence $Z^m$:

	\begin{lemma}\label{Lemma4}
		For any $z\in \R$ the sequence $(Z^{m})_{m\in\N}$ constructed in Lemma \ref{Lemma1} is tight in the Skorokhod topology on $\mathbb D$.
	\end{lemma}
	\begin{proof}
		For the proof we apply Aldous's tightness criterion (see Aldous \cite{Aldous2}). According to Aldous, to prove that  {$\{ Z^{m}:m\in\N\}$} is tight in $\D$ it is enough to show that
	 	\begin{itemize}
   			\item[(i)] for every fixed $t\geq 0$, the set of random variables  {$\{ Z^{m}_t:m\in\N\}$} is tight,
   			\item[(ii)] for every sequence of stopping times $(\tau_m)_{m\in\N}$ (with respect to the filtration $(\cG_t)_{t\geq 0}$) bounded above by $T>0$ and for every sequence of positive real numbers $(\delta_m)_{m\in\N}$ converging to $0$,
 	              		$Z^{m}_{\tau_m +\delta_m} -  Z^{m}_{\tau_m} \to 0$ in probability as $m\to\infty$.
 		\end{itemize}
		To prove (i), by Markov's inequality it is enough to check that, for every fixed $t\geq 0$,
 		\begin{align}\label{help17}
  			\sup_{m\in\N} E\big[  (Z^{m}_t)^2 \big]<\infty.
 		\end{align}	
		Using that $(a+b+c+d+e)^2\leq {5(a^2+b^2+c^2+d^2+e^2)}$ for ${a,b,c,d,e \in\R}$, we obtain that $E\big[(Z^{m}_t)^2\big]$ can be bounded by
 		\begin{align*}
  			&\quad 5z^2  + 5\bigg(\Psi(1)+\int_{-1}^0 (u-1) V(du)\bigg)^2E\bigg[\int_0^t\sign_{(0)}(Z_s^{m}) \,ds\bigg]^2+5\sigma^2 E\bigg[\int_0^t  \sqrt{\big|Z^{m}_s\big|} \,d B_s \bigg]^2\\
			&+5\bigg(\Psi(1)+\int_{-1}^0 (|u|-1) V(du)\bigg)^2E\bigg[\int_0^t \int_{\{-\frac 1 m,\frac 1 m\}} \1_{\{Z^{m}_{s-}=0\}}v\,\mathcal{(M}^m-{\mathcal{M}^m}')(ds,dv)\bigg]^2\\
  			&+5E\bigg[\int_0^t\int_0^{\frac{1}{|Z^{m}_{s-}|}\wedge m} \int_{[-1,1-\frac{1}{m}]}  Z^{m}_{s-}(u-1)(\mathcal{N-N'})(ds,dr,du)\bigg]^2
		\end{align*}
		which, via It\=o's isometry (for the Poissonian integral see for instance page 62 in Ikeda and Watanabe \cite{IW}), can be bounded from above by
	         \begin{align*}
	         		&\quad 5z^2 +{5\bigg(\Psi(1)+\int_{-1}^0 (u-1) V(du)\bigg)^2}t^2+5\sigma^2 E\bigg[\int_0^t |Z^{m}_s|\,\dd s \bigg]\\
			&\quad+5\,t\,\bigg(\Psi(1)+\int_{-1}^0 (|u|-1) V(du)\bigg)^2\int_{\{-\frac 1 m,\frac 1 m\}} v^2 \, \Sigma(dv)\\
			&\quad+5E\bigg[\int_0^t\int_0^{\frac{1}{|Z_s^{m}|}}\int_{[-1,1-\frac{1}{m}]} ( Z_s^{m})^2(u-1)^2\,\dd s\,\dd r\,\bar\Pi(\dd u)\bigg]\\
     			& \leq 5z^2 + 5\bigg(\Psi(1)+\int_{-1}^0 (u-1) V(du)\bigg)^2t^2+5\bigg(\Psi(1)+\int_{-1}^0 (|u|-1) V(du)\bigg)^2 \frac{t}{m}\\
			&\quad+ \bigg(5\sigma^2+5\int_{-1}^1 (u-1)^2\,\bar\Pi(du) \bigg)\int_0^t E[|Z_s^{m}|]\,ds
     		\end{align*}
		so that the estimate $\E[|Z^{m}_s|]\leq 1+\E[(Z_s^{m})^2]$ combined with Gronwall's inequality yields the claim.\\

		Now we turn to (ii) proving the stronger statement that $Z^{m}_{\tau_m+\delta_m} -  Z^{m}_{\tau_m}$ converges to 0 in $L^2$ as ${ m\to\infty}$. Namely, by the SDE \eqref{EquationApprox} and the splitting of summands as before,
		\begin{align*}
   			&\quad E\Big[\big \vert  Z^{m}_{\tau_m+\delta_m} -  Z^{m}_{\tau_m}\big\vert^2 \Big]\\
   			&\leq 4\bigg(\Psi(1)+\int_{-1}^0 (u-1) V(du)\bigg)^2E\bigg[ \int_{\tau_m}^{\tau_m+\delta_m}\sign_{(0)}(Z_s^{m}) \,ds\bigg]^2+ 4\sigma^2 E\bigg[\int_{\tau_m}^{\tau_m+\delta_m}\sqrt{|Z^{m}_s|}\,\dd B_s\bigg]^2\\
			&\quad+4\bigg(\Psi(1)+\int_{-1}^0 (|u|-1) V(du)\bigg)^2E\bigg[\int_{\tau_m}^{\tau_m+\delta_m}\int_{\{-\frac 1 m,\frac 1 m\}}\1_{\{Z^{m}_{s-}=0\}}v\,\mathcal{(M}^m-{\mathcal{M}^m}')(\dd s,\dd v) \bigg]^2\\
  			&\quad + 4E\bigg[\int_{\tau_m}^{\tau_m+\delta_m}\int_0^{\frac{1}{|Z^{m}_{s-}|}\wedge m}\int_{[-1,1-\frac{1}{m}]} Z^{m}_{s-}(u-1)(\mathcal N-\mathcal N')(\dd s,\dd r,\dd u) \bigg]^2.
 		\end{align*}
		The first  summand can be estimated by $C\delta^2_m$ and, hence, can be neglected. By Proposition 3.2.10 in Karatzas and Shreve \cite{KS} we obtain
		\begin{align*}
    			E\bigg[\int_{\tau_m}^{\tau_m+\delta_m}\sqrt{ |Z^{m}_s|}\,\dd B_s\bigg]^2= E\bigg[\int_{\tau_m}^{\tau_m+\delta_m}  |Z^{m}_s| \,\dd s \bigg],
 		\end{align*}
 		and, by Theorem II.1.33 in Jacod and Shiryaev \cite{JS}, the optimal stopping theorem, using the same arguments as in the proof of (3.2.22) in Karatzas and Shreve \cite{KS}, yields
 		\begin{align*}
   			&E\bigg[\int_{\tau_m}^{\tau_m+\delta_m}\int_0^{\frac{1}{|Z^{m}_{s-}|}\wedge m}\int_{[-1,1-\frac{1}{m}]} Z^{m}_{s-}(u-1)(\mathcal N-\mathcal N')(\dd s,\dd r,\dd u) \bigg]^2\\
   			&\quad\leq E\bigg[  \int_{\tau_m}^{\tau_m+\delta_m}\int_0^{\frac{1}{|Z^{m}_s|}}\int_{[-1,1-\frac{1}{m}]} ( Z^{m}_s)^2(u-1)^2 \,\dd s\,\dd r\,\bar\Pi(\dd u) \bigg]\\
                        	&\quad\leq E\bigg[  \int_{\tau_m}^{\tau_m+\delta_m}|Z^{m}_s| \,\dd s\,\bigg] \int_{-1}^1 (u-1)^2\bar\Pi(\dd u).
 		\end{align*}
		For the integral with respect to $\mathcal M$ a similar arguments gives the upper bound $C_3 \delta_m$.
		In total this shows that (we can suppose that $\delta_m\leq 1$),
 		\begin{align*}
			E\big[ \vert  Z^{m}_{\tau_m+\delta_m} -  Z^{m}_{\tau_m}\vert^2 \big]
			&\leq C_1\delta_m^2+C_2 \delta_m+C_3E\bigg[  \int_{\tau_m}^{\tau_m+\delta_m}|Z^{m}_s| \,\dd s\,\bigg]\\
			&\leq C_1\delta_m^2+C_2\delta_m+C_3 \delta_m E\bigg[\sup_{t\leq T+1} |Z^{m}_t|\bigg].
    		\end{align*}
   		Hence, the proof is complete if we can show that $E\big[\sup_{t\leq T+1} |Z^{m}_s|\big]$ is bounded in $m$. But this follows easily with the same arguments exploited for (i) using the Burkholder-Davis-Gundy inequality instead of It\=o's isometry.	\end{proof}

	We next prepare for the convergence proof of $Z^{m}$ along subsequences. It is crucial to deduce, a priori, that all limiting points do not spend time at zero in order to control the discontinuity of the sign-function at zero. 	Since we cannot deduce this property for the limiting points of $Z^{m}$ directly, we show it for $|Z^{m}|$ which is substantially simpler.
	
	\begin{lemma}\label{Lemma3}
		Suppose $Z$ denotes a limiting point of the tight sequence $(Z^{m})_{m\in\N}$ constructed in Lemma \ref{Lemma1}, then $|Z|$ is a weak solution to the SDE
		\begin{align}\label{H1}
			\begin{split}
			X_t&=z+\bigg(\Psi(1)+\int_{-1}^0 (|u|-1) V(du)\bigg)\,t+\sigma\int_0^t\sqrt{X_s}\,dB_s\\
			&\quad+\int_0^t \int_0^{\frac{1}{X_{s-}}}\int_{-1}^1X_{s-}(|u|-1)(\mathcal{N-N'})(ds,dr,du).
			\end{split}
		\end{align}
	\end{lemma}
	\begin{proof}
		Let us suppose that along the subsequence $m_k$ we have weak convergence of $Z^{m_k}$ to $Z$ and, due to the continuous mapping theorem, also weak convergence of $|Z^{m_k}|$ to $|Z|$. We first derive a martingale problem for 
		$|Z^{m}|, m\in\N,$ from which we then derive the claimed statement.
		\smallskip
		
		\textbf{Step A):} Proceeding exactly as in Step 1) of the proof of Proposition \ref{Proposition}, one derives from Lemma \ref{Lemma7} that $|Z^{m}|$ solves the martingale problem $(|\mathcal A^m|,\delta_{|z|})$ with
 		\begin{align*}
  			&\quad(|\mathcal A^{m}|f)(x)\\
  			&:=\bigg(\bigg(\Psi(1)+\int_{-1}^0 (u-1) V(du)\bigg)\\
			&\qquad+\bigg(\int_{-1}^0 (|u|-u) V(du)\bigg)\big(\1_{\{x=0\}}+(1\wedge m x)\big)\bigg)f'(x)\\
			&\quad+\frac{\sigma^2}{2}xf''(x)+\int_{\{\pm\frac 1 m\}} \1_{\{x=0\}} \bigg(f\bigg(\bigg(\Psi(1)+\int_{-1}^0 (|u|-1) V(du)\bigg)|u|\bigg)\\
			&\quad\quad\quad-f(0)-f'(0)\bigg(\Psi(1)+\int_{-1}^0 (|u|-1) V(du)\bigg)|u|\bigg) \Sigma(dv)\\
      	   	         &\quad + \int_0^{m} \int_{[-1,1-\frac{1}{m}]} \mathbf 1_{\{rx\leq 1\}}\Big(f (|u|x)-f(x) -f'(x)x(|u|-1)\Big) \,\dd r\,\bar\Pi(\dd u),\quad x\geq 0,
 		\end{align*}
		for $f\in C_c^\infty[0,\infty)$.
 		\smallskip

		\textbf{Step B):} Next, we show that the limit points $|Z|$ solve the martingale problem $(|\mathcal A|,\delta_{|z|})$, with 
 		\begin{align*}
  			(|\mathcal A| f)(x)
  			&:=\bigg(\Psi(1)+\int_{-1}^0 (|u|-1) V(du)\bigg)\,f'(x)+\frac{\sigma^2}{2}xf''(x)\\
      	   	         &\quad + \int_0^{\infty} \int_{[-1,1]}\mathbf 1_{\{rx\leq 1\}}\Big(f (|u|x)-f(x) -f'(x)x(|u|-1)\Big) \,\dd r\,\bar\Pi(\dd u),\quad x\geq 0,
 		\end{align*}
		for $f\in C_c^\infty[0,\infty)$. By Skorokhod's representation theorem we may assume that $ Z^{m_k}$ converges to $Z$ (resp. $|Z^{m_k}|$ to $|Z|$) almost surely in $\D$ (possibly on a different probability space and changing also the 
		subsequence $(m_k)_{k\in\N}$). By Proposition 3.5.2 of \cite{EthierKurtz}, the almost sure convergnece yields that $P(\bar\Omega)=1$, where
		\begin{align*}
    			\bar\Omega:=\left\{\omega\in\Omega : \lim_{k\to\infty} |Z^{m_k}_t(\omega)| = |Z_t(\omega)| \quad \text{for $t\geq 0$ at which $(Z_u(\omega))_{u\geq 0}$ is continuous}\right\}.
		\end{align*}
		In what follows we let $\omega\in\bar\Omega$ be fixed and show that
 		\begin{align}\label{help18}
   			\lim_{k\to\infty} \int_0^t (|\mathcal A^{m_k}|f)( |Z^{m_k}_s|)(\omega)\,\dd s= \int_0^t (|\mathcal A| f)( |Z_s|)(\omega)\,\dd s,\qquad t\geq 0.
 		\end{align}
		In Step B1) we verify the pointwise convergence $(|\mathcal A^{m_k}|f)( |Z^{m_k}_s|)(\omega)\stackrel{m_k\to\infty}{\longrightarrow} (|\mathcal A| f)( |Z_s|)(\omega)$ for $s\leq t$ fixed and in Step B2) we verify the convergence of \eqref{help18} via dominated convergence.\\
		 Let us introduce the notation
		\begin{align*}
   			D(\omega):=\big\{ t\geq 0 : \text{$(Z_u(\omega))_{u\geq 0}$ is continuous at $t$} \big\},
 		\end{align*}
		so that $\lim_{k\to\infty} |Z^{m_k}_t(\omega)| = |Z_t(\omega)|$ for all $t\in D(\omega)$ and furthermore $[0,\infty)\setminus D(\omega)$ is at most countable since $Z$ has c\`{a}dl\`{a}g paths.
		\smallskip

		\textbf{Step B1a):} The pointwise convergence for the drift part is trivial.
		\smallskip

		\textbf{Step B1b):} The pointwise convergence for the diffusive part is trivial.
		\smallskip
		
		\textbf{Step B1c):} For the integral with respect to $\Sigma$ we apply Taylor's formula to find
		\begin{align*}
			&\quad\int_{\{\pm\frac 1 m\}} \1_{\{Z^{m_k}_s=0\}} \bigg(f\bigg(\bigg(\Psi(1)+\int_{-1}^0 (|u|-1) V(du)\bigg)|v|\bigg)\\
			&\quad\quad\quad-f(0)-f'(0)\bigg(\Psi(1)+\int_{-1}^0 (|u|-1) V(du)\bigg)|v|\bigg) \Sigma(dv)\\
			&\leq \frac{1}{2}\bigg(\Psi(1)+\int_{-1}^0 (|u|-1) V(du)\bigg)^2\sup_z f''(z) \int_{\{-\frac{1}{m_k},\frac{1}{m_k}\}} |v|^2 \, \Sigma(dv)\\
			&=\frac{1}{2}\bigg(\Psi(1)+\int_{-1}^0 (|u|-1) V(du)\bigg)^2\frac{\sup_z f''(z)}{m_k},
		\end{align*}
		so that pointwise convergence to zero for $m_k\to\infty$ is verified.
		\smallskip
		
		\textbf{Step B1d):}
		For the integral with respect to $\bar\Pi$ we use that, for all $s\in D(\omega)$, the integrand
		\begin{align*}
 			(r,u)&\mapsto  \mathbf 1_{\{u\in [-1,1-\frac{1}{m}]\}} \mathbf 1_{\{r|Z^{m_k}_s(\omega)|\leq 1\}}\\
          	         \quad&\times \big[ f\big(|u||Z^{m_k}_s(\omega)|\big)
	         		-f(|Z^{m_k}_s(\omega)|) - f'(|Z^{m_k}_s(\omega)|)|Z^{m_k}_s(\omega)|(|u|-1)\big]
		\end{align*}
 		converges, as $k\to\infty$, pointwise to
 		\begin{align*}
 		(r,u)\mapsto	\1_{\{|u|\leq 1\}}\mathbf 1_{\{r|Z_s(\omega)|\leq 1\}} \big[  f(|u||Z_s(\omega)|)-f(|Z_s(\omega)|)-f'(|Z_s(\omega)|)|Z_s(\omega)|(|u|-1)\big].
 		\end{align*}
		
		\textbf{Step B2):} Since
 		\begin{align*}
    			\lim_{k\to\infty}(|\mathcal A^{m_k}|f)(| Z^{m_k}_s(\omega)|)=(|\mathcal A| f)( |Z_s(\omega)|),\quad \omega \in \bar\Omega,
		\end{align*}
		is verified for any $s\leq t$, it remains to justify the change of limit and integration in \eqref{help18}. For the first threee summands dominated convergence is clear (in Step B1c) the upper bound is independent of $s$) and we only need to deal with the integral with respect to $\bar\Pi$:\\
		By Taylor expansion of second order, we can derive the upper bound
		\begin{align*}
 			\begin{split}
  				&\quad\sup_{k\in\N}\mathbf 1_{u\in \{[-1,1-\frac{1}{m}]\}} \mathbf 1_{\{r |Z^{m_k}_s(\omega)|\leq 1\}}\\
				 &\qquad \Big\vert f(|u||Z^{m_k}_s(\omega)|)-f( |Z^{m_k}_s(\omega)|)\mathbf 1_{\{\vert u\vert\geq \eps_k\}}- f'( Z^{m_k}_s(\omega))|Z^{m_k}_s(\omega)|(|u|-1)\Big\vert\\
				&\leq \frac{1}{2}\1_{\{|u|\leq 1\}}\sup_z \vert f''(z)\vert\sup_{k\in\N} \mathbf 1_{\{r |Z^{m_k}_s(\omega)|\leq 1\}}|Z^{m_k}_s(\omega)|^2 (|u|-1)^2\\
				&\leq \frac{1}{2}\1_{\{|u|\leq 1\}}\sup_z \vert f''(z)\vert\sup_{k\in\N}  \left(Z^{m_k}_s(\omega)\wedge\frac{1}{r}\right)^2 (|u|-1)^2\\
				&\leq \frac{1}{2}\1_{\{|u|\leq 1\}}\sup_z \vert f''(z)\vert\sup_{k\in\N}  \left(\sup_{s\leq t}Z^{m_k}_s(\omega)\wedge\frac{1}{r}\right)^2 (|u|-1)^2.
			\end{split}
 		\end{align*}
		Note that for the final line we used that $x\mapsto \sup_{s\leq t}x_s$ is a continuous functional on the Skorokhod space so that the convergence of $Z^{m_k}$ implies $\sup_{k\in\N}\sup_{s\leq t}Z^{m_k}_s(\omega)=: C_t(\omega)<\infty$.
		Thus, the integral with respect to $\bar\Pi$ in $(\mathcal A^{m_k}f)(|Z^{m_k}_s(\omega)|)$ is bounded from above by
  		\begin{align*}
   			\begin{split}
    				&\quad \frac{1}{2}\sup_z\vert f''(z)| \int_0^\infty \left(C_t(\omega)\wedge \frac{1}{r}\right)^2  dr\int_{-1}^1(|u|-1)^2\,\bar\Pi(\dd u)\\
				&\leq C\sup_z \vert f''(z)|\bigg(C_t(\omega)^2+\int_1^\infty r^{-2}dr\bigg)\int_{-1}^1(|u|-1)^2\,\bar\Pi(\dd u),
  			\end{split}
 		\end{align*}
		which is finite and independent of $s$. Using dominated convergence theorem we have convergence for the third summand of $\mathcal A^{m_k} f$.
\smallskip

 		\textbf{Step C):} To conclude the proof let us write 
		\begin{align*}
			M^{m_k}_t&=f(Z^{m_k}_t)-\int_0^t \mathcal |\mathcal A^m|f(Z^{m_k}_s)\,ds,\quad t\geq 0,\\
			M_t&=f(Z_t)-\int_0^t \mathcal |\mathcal A|f(Z_s)\,ds,\quad t\geq 0,
		\end{align*}
		for which we know that the $M^{m_k}$ are martingales with respect to the filtrations generated by $Z^{m_k}$. The martingale property of $M$ with respect to its own filtration follows by Jacod and Shiryaev \cite[Corollary IX.1.19]{JS}. To check the conditions of this result we have to show that $M^{m_k}$ converges weakly in $\D$ as 
		$k\to\infty$ to $M$ and that there is some $b\geq 0$ such that $\vert \Delta M^{m_k}_t\vert\leq b$ for all $t>0$, $m\in\N$, almost surely. Using that $Z^{m_k}$ converges weakly in $\D$ to $Z$ as $k\to\infty$ and that $f$ is continuous
 		and bounded, we have $f(Z^{m_k})$ converges weakly in $\D$ to $f(Z)$ as $k\to\infty$. Since the integral in the definition of $M$ is continuous, by Jacod and Shiryaev \cite[Proposition VI.1.23]{JS}, we obtain that $M^{m_k}$ converges weakly 		in $\D$ as $k\to\infty$ to $M$. Further, almost surely for all $t\geq 0$,
 		\begin{align*}
   			\vert \Delta M^{m_k}_t \vert=\vert f(|Z^{m_k}_t|) - f(|Z^{m_k}_{t-}|)\vert\leq 2 \sup_z\vert f(z)\vert<\infty.
 		\end{align*}
	\end{proof}

\begin{cor}\label{Cor1}
	Suppose $Z$ denotes a limiting point of the tight sequence $(Z^{m})_{m\in\N}$ constructed in Lemma \ref{Lemma1}, then $Z$ almost surely does not spend time at zero.
\end{cor}
\begin{proof}
	Utilizing Lemma \ref{Lemma3} it is enough to show that any non-negative weak solution to the SDE \eqref{H1} does not spend time at zero.
	Without further assumptions on the jump measure $\bar\Pi$ the jumps of a solution $X$ to the SDE \eqref{H1} are not summable, thus, we cannot directly resort to a simple local time argument based on the occupation time formula (compare for instance Section IV.6 of \cite{Protter}). Instead, we use an It\=o formula argument that was used in a more specific situation in \cite{BDMZ}.\\ The argument is based on the trivial fact $\sqrt{X^2_t}=X_t$ and a double use of It\=o's formula, once applied to a smooth function and once to a singular function. The singular use gives an additional term from which the claim follows. Here is the simple direction applying It\=o's formula to the $C^2([0,\infty))$-function $f(x)=x^2$:
	\begin{align*}
		X_t^2
		&=z^2+\bigg(2\bigg(\Psi(1)+\int_{-1}^0 (|u|-1) V(du)\bigg)+\sigma^2+\int_{-1}^1(u^2-2|u|+1)\,\bar\Pi(du)\bigg) \int_0^t  X_s \,ds\\
		&\quad+2\sigma\int_0^t  X_s^{3/2}\,dB_s +\int_0^t \int_0^{\frac{1}{X_{s-}}}\int_{-1}^1X_{s-}^2(u^2-1)\,\mathcal{(N-N')}(ds,dr,du),\quad t\geq 0.
	\end{align*}
	Next, we proceed with the inverse direction. Suppose we could apply It\=o's formula with $f(x)=\sqrt{x}$ and the convention $\frac{0}{0}=0$ to the semimartingale decomposition derived for $X_t^2$, then 
	\begin{align}\label{H2}\begin{split}
		X_t&=z+\bigg(\Psi(1)+\int_{-1}^0 (|u|-1) V(du)\bigg)\int_0^t \1_{\{X_s> 0\}}\,ds+\sigma\int_0^t \sqrt{X_s}\,dB_s\\
		&\quad+\int_0^t \int_0^{\frac{1}{X_{s-}}}\int_{-1}^1X_{s-}(|u|-1)(\mathcal{N-N'})(ds,dr,du),\quad t\geq 0,
	\end{split}
	\end{align}
	so that comparing the drifts of \eqref{H2} and \eqref{H1} implies the claim. To verify Equation \eqref{H2} rigorously, we approximate $f(x)=\sqrt{x}$ on $[0,\infty)$ by the $C^2([0,\infty))$-functions $f^\eps(x)=\sqrt{x+\eps}$. Applying It\=o's formula to the semimartingale decomposition derived for $X_t^2$ gives
	\begin{align*}
		&\quad \sqrt{X_t^2+\eps}\\
		&=\sqrt{z^2+\eps}+\bigg(2\bigg(\Psi(1)+\int_{-1}^0 (|u|-1) V(du)\bigg)+\sigma^2+\int_{-1}^1(u^2-2|u|+1)\,\bar\Pi(du)\bigg)\frac{1}{2} \int_0^t ( X_s^2+\eps)^{-\frac{1}{2}}X_s \,ds\\
		&\quad+\sigma\int_0^t ( X_s^2+\eps)^{-\frac{1}{2}} X_s^{3/2}\,dB_s-\frac{\sigma^2}{2}\int_0^t ( X_s^2+\eps)^{-\frac{3}{2}} X_s^3\,ds\\
		&\quad+\int_0^t \int_0^{\frac{1}{X_{s-}}}\int_{-1}^1\Big(\sqrt{X_{s-}^2u^2+\eps}-\sqrt{X_{s-}^2+\eps}\Big)\,\mathcal{(N-N')}(ds,dr,du)\\
		&\quad+\int_0^t \int_0^{\frac{1}{X_s}}\int_{-1}^1\Big(\sqrt{X_s^2u^2+\eps}-\sqrt{X_s^2+\eps}-\frac{1}{2}(X_s^2+\eps)^{-1/2}X_s^2(u^2-1)\Big)\,ds\,dr\,\bar\Pi(du)\\
		&=:\sqrt{z^2+\eps}+I^{1,\eps}_t+I^{2,\eps}_t+I^{3,\eps}_t+I^{4,\eps}_t+I^{5,\eps}_t.
	\end{align*}
	Since the left-hand side converges to $X_t$ almost surely, it suffices to find a subsequence $\eps_k$ along which the summands $I^{1,\eps_k}_t,...,I^{5,\eps_k}_t$ converge almost surely to the summands of \eqref{H2}.
	\smallskip
	
	For the drift we directly obtain the almost sure convergence
	\begin{align*}
		I^{1,\eps}_t
		\stackrel{\eps\to 0}{\longrightarrow}\bigg(\bigg(\Psi(1)+\int_{-1}^0 (|u|-1) V(du)\bigg)+\frac{\sigma^2}{2}+\frac{1}{2}\int_{-1}^1(u^2-2|u|+1)\,\bar\Pi(du)\bigg)  \int_0^t \1_{\{X_s> 0\}}\,ds
	\end{align*}
	by dominated convergence. To show convergence of $I^{2,\eps}$ we first use It\=o's isometry to obtain
	\begin{align}\label{yi}
		E\bigg[\bigg(I_t^{2,\eps}-\sigma\int_0^t \sqrt{X_s}\,dB_s\bigg)^2\bigg]&= \sigma^2E\bigg[\int_0^t\big(( X_s^2+\eps)^{-\frac{1}{2}}X_s^{3/2}-X_s^{1/2}\big)^2\,ds\bigg].
	\end{align}	
 If we define $g_\eps(x)= \big(( x^2+\eps)^{-\frac{1}{2}}x^{3/2}-x^{1/2}\big)^2$, then 
 \begin{align*}
 	\frac{\partial }{\partial \eps }g_\eps(x)=- \big(( x^2+\eps)^{-\frac{1}{2}}x^{3/2}-x^{1/2}\big)(x^2+\eps)^{-3/2} x^{3/2}\geq  0,\quad x\geq 0.
 \end{align*}
 Since $g_\eps(x)$ converges pointwise to zero as $\eps$ tends to zero,  the righthand side of \eqref{yi} converges to zero by monotone convergence so that $I_t^{2,\eps}$ converges to $\sigma\int_0^t \sqrt{X_s}\,dB_s$ in $L^2$. The almost sure convergence 
		\begin{align*}
		I_t^{3,\eps}\stackrel{\eps\to 0}{\longrightarrow}-\frac{\sigma^2}{2}\int_0^t \1_{\{X_s> 0\}}\,ds
	\end{align*}
	is proved as in Step 1) and cancels in the limit the second summand of $I^{1,\eps}$. To show convergence of  $I_t^{4,\eps}$ we use the It\=o isometry for Poissonian integrals to find
	\begin{align*}
		&\quad E\bigg[\bigg(I_t^{4,\eps}-\int_0^t \int_0^{\frac{1}{X_{s-}}}\int_{-1}^1X_{s-}(|u|-1)(\mathcal{N-N'})(ds,dr,du)\bigg)^2\bigg]\\
		&=\E\bigg[\int_0^t\ \int_{-1}^1 \frac{1}{X_s}\Big(\sqrt{X_{s}^2u^2+\eps}-\sqrt{X_{s}^2+\eps}-X_s(|u|-1)\Big)^2\,ds\,\bar\Pi(du)\bigg].
	\end{align*}
	With $h_\eps(x)= \frac{1}{x}\big(\sqrt{x^2u^2+\eps}-\sqrt{x^2+\eps}-x(|u|-1)\big)^2$ we claim that
	\begin{align*}
		\frac{\partial}{\partial \eps}h_\eps(x)&=-\frac{1}{x}\big(\sqrt{x^2u^2+\eps}-\sqrt{x^2+\eps}-x(|u|-1)\big)\big((x^2u^2+\eps)^{-1/2}-(x^2+\eps)^{-1/2}\big)\leq 0
	\end{align*}
	for all $x\geq 0$ and $|u|\leq 1$. To see the claim note that the second bracket is clearly positive so that it suffices to show that the first bracket is positive. For $\eps=0$ the first bracket is zero and it is easy to see that the first derivative in $\eps$ is positive. 	Hence, the $L^2$-convergence of $I_t^{4,\eps}$ to $\int_0^t \int_0^{\frac{1}{X_{s-}}}\int_{-1}^1X_{s-}(u-1)(\mathcal{N-N'})(ds,dr,du)$ follows again from monotone convergence. 
	\smallskip
	
	Finally, if we rewrite $I^{5,\eps}_t$ as
	\begin{align*}
		I^{5,\eps}_t&=\int_0^t \int_{-1}^1\frac{1}{X_s}\Big(\sqrt{X_s^2u^2+\eps}-\sqrt{X_s^2+\eps}\Big)\,ds\,\bar\Pi(du)\\
		&\quad-\int_0^t\int_{-1}^1\frac{1}{2}(X_s^2+\eps)^{-1/2}X_s(u^2-1)\,ds\,\bar\Pi(du),
	\end{align*}
	then the first summand converges by monotone convergence shown as above and the second summand by dominated convergence:
	\begin{align*}
		I_t^{5,\eps}\stackrel{\eps\to 0}{\longrightarrow}\int_{-1}^1\Big(|u|-1-\frac{1}{2}(u^2-1)\Big)\bar\Pi(du)\int_0^t \1_{\{X_s> 0\}}\,ds,\quad a.s.,
	\end{align*}
	so that the third summand of $I^{1,\eps}$ is cancelled in the limit.
	\smallskip
	
	Choosing a common subsequence $\eps_k$ such that all terms converge almost surely the proof can be completed.
\end{proof}

With the previous lemma we know that under condition \eqref{Condition} solutions do not become trivial (trapped at zero). Additionally we can circumvent two major problems of the construction: First, the problem of the discontinuity of the sign-function at zero is resolved. Secondly, our redefinition of $\sign_{(0)}$ in the approximating equations to ensure symmetry of the approximating sequence does not pose any problem since it is not seen by the limiting process $Z$.

\begin{lemma}\label{existence}
		Suppose $Z$ denotes a limiting point of the tight sequence $(Z^{m})_{m\in\N}$ constructed in Lemma \ref{Lemma1}, then $Z$ is a weak solution to \eqref{Equation}.
\end{lemma}
\begin{proof}
	Taking into account Proposition \ref{Proposition} it suffices to show that any weak limiting point $Z$ of the tight sequence $(Z^m)_{m\in\N}$ constructed in Lemma \ref{Lemma1} satisfies the martingale problem $(\mathcal A,\delta_z)$. The proof follows along the same lines as the proof of Lemma \ref{Lemma3} changing the state-space from $[0,\infty)$ to $\R$ and the generators to
	 	\begin{align*}
  			&\quad(\mathcal A^{m}f)(x)\\
  			&:=\bigg(\Psi(1)+\int_{-1}^0 (u-1) V(du)\bigg)\sign_{(0)}(x)f'(x)
			+\frac{\sigma^2}{2}|x|f''(x)\\
			&\quad+\int_{\{\pm \frac{1}{m}\}} \1_{\{x=0\}} \bigg[f\bigg(\bigg(\Psi(1)+\int_{-1}^0 (|u|-1) V(du)\bigg)v\bigg)\\
			&\quad\quad\quad -f(0)-f'(0)\bigg(\Psi(1)+\int_{-1}^0 (|u|-1) V(du)\bigg)v\bigg] \Sigma(dv)\\
      	   	         &\quad + \int_0^{m} \int_{[-1,1-\frac{1}{m}]} \mathbf 1_{\{r|x|\leq 1\}}\Big(f (ux)-f(x) -f'(x)x(u-1)\Big) \,\dd r\,\bar\Pi(\dd u),\quad f\in C_c^\infty((-\infty,\infty)),
 		\end{align*}
		 and
	 	\begin{align*}
  			&\quad(\mathcal Af)(x)\\
  			&:=\bigg(\Psi(1)+\int_{-1}^0 (u-1) V(du)\bigg)\sign(x)f'(x)+\frac{\sigma^2}{2}|x|f''(x)\\
      	   	         &\quad + \int_0^{\infty} \int_{[-1,1]}\mathbf 1_{\{r|x|\leq 1\}}\Big(f (ux)-f(x) -f'(x)x(u-1)\Big) \,\dd r\,\bar\Pi(\dd u),\quad f\in C_c^\infty((-\infty,\infty)).
 		\end{align*}
		Comparing with the proof of Lemma \ref{Lemma3}, the ony difference occurs in Step B1a) because the pointwise convergence for the drift coefficients fails since (i) the sign-function is defined differently for the approximating 		martingale problem and the limit martingale problem and (ii) both sign-functions are discontinuous. Both problems are avoidable via Corollary \ref{Cor1} applied for the final step of
		\begin{align*}
			\lim_{m\to\infty}		\int_0^t\sign_{(0)}\big(Z^{m}_s(\omega)\big) \,ds
			&=\lim_{m\to\infty}	\int_0^t\sign_{(0)}\big(Z^{m}_s(\omega)\big)  \1_{\{Z_s(\omega)\neq 0\}}\,ds\\
			&=\int_0^t \sign_{(0)}\big(Z_s(\omega)\big)  \1_{\{Z_s(\omega)\neq 0\}}\,ds\\
			&=\int_0^t\sign\big(Z_s(\omega)\big)\1_{\{Z_s(\omega)\neq 0\}} \,ds\\
			&=\int_0^t\sign\big(Z_s(\omega)\big) \,ds.
		\end{align*}
		For the pointwise convergence of the integrands we used the continuity of the sign-function away from zero and dominated convergence to interchange limits and integration. 
	The proof is now complete.
\end{proof}
Now that we have constructed processes $Z^z$ started from $z\in\R$ that are weak solutions to \eqref{Equation} and symmetric by construction, we need to show that the family $(Z^z)_{z\in\R}$ is 
\begin{itemize}
	\item[(i)] Markovian,
	\item[(ii)] self-similar.
\end{itemize}
Both statements are derived from a weak uniqueness statement for \eqref{Equation} for which we had to impose Assumption \textbf{(A)}.\\
For initial condition zero, we derive moment equations for \eqref{Equation} from which, due to Assumption \textbf{(A)}, the well-posedness of the moment problem for one-dimensional marginals of symmetric solutions can be deduced. For initial conditions different from zero pathwise uniqueness before hitting zero holds. Combining the two uniqueness statements, uniqueness for one-dimensional marginals of solutions issued from the same initial condition follows. The Markov property is then a consequence of martingale problem theory and the self-similarity can be deduced from the self-similar structure of the coefficients in \eqref{Equation}.
\begin{proposition}\label{Markov}
		Denote by $Z$ a limiting point of the tight sequence $(Z^{m})_{m\in\N}$ with initial conditions $Z^m_0=z$ constructed in Lemma \ref{Lemma1}, then $Z$ is Markovian.
\end{proposition}
\begin{proof}
	Since we showed that $Z$ is a weak solution to the SDE \eqref{Equation} we start with a weak uniqueness statement for solutions to the SDE \eqref{Equation} that satisfies the symmetry property
\begin{align}\label{sym2}
	P(X_{T_0+t}\in A)=P(X_{T_0+t}\in -A),\quad t\geq 0, A\in \mathcal B(\R),
\end{align}
and do not spend time at zero. Both properties hold for $Z$: the first by construction as weak limit of the symmetric $Z^m$ defined by \eqref{EquationApprox}, the latter by Corollary \ref{Cor1}.
	\smallskip
	
	\textbf{Step 1a):} Let us first deduce the almost sure dichotomy $T_0<\infty$ or $T_0=\infty$ for the first hitting time of zero. A singular application of It\=o's formula as in the proof of Corollary \ref{Cor1} yields the 		semimartingale decomposition
		\begin{align}\label{absv}\begin{split}
			|X_t|			&=|z|+\bigg(\Psi(1)+\int_{-1}^0 (|u|-1) V(du)\bigg)\,t\\
			&\quad+\sigma\int_0^t \sign(X_s)\sqrt{|X_s|}\,dB_s+\int_0^t\int_0^{\frac{1}{|X_{s-}|}}\int_{-1}^1 |X_{s-}|(|u|-1)(\mathcal{N-N'})(ds,dr,du)\end{split}
		\end{align}
		for $t\geq 0$. Replacing $B$ by the Brownian motion $\bar B_t=\int_0^t \sign(X_s)dB_s$, we find that \eqref{absv} coincides with \eqref{SDELeifMatyas} so that $|X|$ is a positive self-similar Markov process. Since the first hitting times of $X$ and $|X|$ coincide, the claimed dichotomy follows from Lamperti's dichotomy (Section 3 of \cite{L}) for positive self-similar Markov processes.
\smallskip

\textbf{Step 1b):} Pathwise uniqueness holds for \eqref{Equation} up to first hitting zero: Suppose that $X^1, X^2$ are two solutions driven by the same noises $B$ and $\mathcal N$ and set $T_{\frac 1 n}=\inf\{ t\geq 0: |X^1_t|\leq \frac 1 n\text{ or }| X^2_t|\leq \frac 1 n\}$ for $\frac{1}{n}<|X^i_0|$. Then,
		\begin{align*}
			P(X^1_t=X^2_t\text{ for all }t<T_{\frac 1 n})=1
		\end{align*}
		since all integrands are locally Lipschitz continuous away from zero. Letting $n$ tend to infinity and using the right-continuity of solutions we find that
		\begin{align*}
			P(X^1_t=X^2_t\text{ for all }t<T_{0})=1
		\end{align*}
		and in particular that 
		\begin{align*}
			T_0=\inf\{ t \geq 0 : X^1_t=0\}=\inf\{ t \geq 0 : X^2_t=0\}.
		\end{align*}
		In what follows we assume $T_0<\infty$ almost surely; otherwise, we can directly proceed with Step 2) since pathwise uniqueness implies uniqueness in law.
		\smallskip
		
\textbf{Step 1c):} The pathwise uniqueness implies that solutions to \eqref{Equation} are strong up to $T_0$ and, consequently, $T_0$ is a stopping time for $B$ and $\mathcal N$. We denote by $\tilde B$ and $\tilde N$ the noises shifted by $T_0$. Due to the strong Markov property of the Brownian motion and the Poisson point process $\tilde B$ is a Brownian motion and $\tilde{\mathcal N}$ a Poisson point process with same intensity as $\mathcal N$. Furthermore, we define the shifted process $(\tilde X_t=X_{T_0+t})_{t\geq 0}$ that satisfies weakly the SDE
		\begin{align}\label{dd}
		\begin{split}
			\tilde X_t&=\bigg(\Psi(1)+\int_{-1}^0 (u-1) V(du)\bigg)\int_{0}^{t}\sign(\tilde X_{s})\,ds+\sigma\int_{0}^{t} \sqrt{|\tilde X_{s}|}\,d\tilde B_s\\
			&\quad+\int_0^{t} \int_0^{\frac{1}{|\tilde X_{s-}|}}\int_{-1}^1 \tilde X_{s-}(u-1)(\tilde {\mathcal N}-\tilde{\mathcal{N}}')(ds,dr,du) ,\quad t\geq 0.
		\end{split}
		\end{align}		
		In other words, $\tilde X$ is a weak solution to the SDE \eqref{Equation} with respect to the noises $\tilde B$ and $\tilde{\mathcal N}$ issued from zero. Furthermore, the symmetry condition \eqref{sym2} implies the symmetry condition
\begin{align}\label{symc}
	P(\tilde X_{t}\in A)=P(\tilde X_{t}\in -A),\quad t\geq 0, A\in \mathcal B(\R),
\end{align}
and clearly $\tilde X$ does not spend time at zero since $X$ does not.
\smallskip

\textbf{Step 1d):} Next, we need that all solutions to the SDE \eqref{dd} that do not spend time at zero and satisfy the symmetry condition \eqref{symc} have the same one-dimensional marginals. The symmetry assumption implies that the one-dimensional laws $\tilde X_t$ are uniquely determined by $|\tilde X_t|$. Since $|\tilde X_t|$ satisfies the SDE \eqref{absv} with the noises $\tilde B$ and $\tilde{\mathcal N}$ it suffices to show that the moment problem for \eqref{absv} is well-posed; this follows from the main result of \cite{BD} since the jumps of
\begin{align*}
t\mapsto \int_0^t\int_0^{\frac{1}{|\tilde X_{s-}|}}\int_{-1}^1 |\tilde X_{s-}|(|u|-1)(\mathcal{\tilde N-\tilde{N'}})(ds,dr,du)
\end{align*}
 are negative (this is the reason for our Assumption \textbf{(A)}). It was shown in \cite{BD} that the $k$th moments equal $C_k t^k$ for some constants $C_k$ that decrease sufficiently fast so that the moment problem is well-posed. Furthermore, the $C_k$ only depend on $(a,\sigma^2,\Pi,q,V)$ but not on the solution, thus, well-posedness of the moment problem implies the uniqueness of one-dimensional marginals.

\smallskip

\textbf{Step 1e):}
	Let us now suppose $X^1$ and $X^2$ are two weak solutions to \eqref{Equation} that do not spend time at zero and both satisfy the symmetry condition \eqref{sym2}. We split according to
		\begin{align*}
			P(X_t^i\in A)=P(X_t^i\in A\, ,\, t\leq T_0^i)+P(X_t^i\in A\,,\,t>T_0^i)
		\end{align*}
	and show that 
	\begin{align}
		P(X_t^1\in A\, ,\, t\leq T_0^1)&=P(X_t^2\in A\, ,\, t\leq T_0^2),\label{11}\\
		P(X_t^1\in A\,,\,t>T_0^1)&=P(X_t^2\in A\,,\,t>T_0^2).\label{22}
	\end{align}
	Equality \eqref{11} follows from the pathwise uniqueness before hitting zero so that we only need to verify \eqref{22}. Using the defining equation for the $X^i$ and the definition of $\tilde X^i$ above, one can rewrite
	\begin{align*}
		P(X_t^i\in A\,,\,t>T_0^i)=P(\tilde X^i_{t-T_0^i}\in A\,,\,t>T_0^i).
	\end{align*}
	Integrating out $P(T_0^i\in ds)$ (note that $P(T_0^1\in ds)=P(T_0^2\in ds)$ have same law as shown in Step 1b)) we obtain \eqref{22} from Step 1d) since $\tilde X^i_0=0$.
\smallskip

\textbf{Step 2):} The uniqueness of one-dimensional marginals for weak solutions to \eqref{Equation} that spend zero time at zero and satisfy the symmetry condition \eqref{symc} now implies the Markov property for the weak solution $Z$ by martingale problem arguments such as in the proof of Theorem 4.4.2 of \cite{EthierKurtz}. The required measurability $z\mapsto P^z$ is a consequence of the construction: the measurability (even continuity) in the initial condition holds for the $Z^m$ and since the pointwise limit of measurable functions remains measurable, the measurability for the limit follows.

	\end{proof}

	\begin{proposition}\label{Pro2}
		Denote by $Z^z$ a limiting point of the tight sequence $(Z^{m})_{m\in\N}$ with initial conditions $z\in \R$ constructed in Lemma \ref{Lemma1}. Then the family $(Z^z)_{z\in\R}$ is a real-valued self-similar Markov family with Lamperti-Kiu quintuple $(a,\sigma^2,\Pi,q,V)$.
	\end{proposition}
	
	\begin{proof}
	 For $c>0$ fixed we define $\bar Z^z_t:=\frac{1}{c}Z^z_{ct}$, $t\geq 0$. Since $Z$ is a weak solution to \eqref{Equation}, $\bar Z^z$ satisfies
	\begin{align*}
		\bar Z^z_t&=\frac{z}{c}
                       + \bigg(\Psi(1)+\int_{-1}^0 (u-1) V(du)\bigg) \frac{1}{c}\int_0^{ct} \sign(Z^z_s)\,ds
                       + \frac{\sigma}{c} \int_0^{ct}\sqrt{|Z^z_s|} \,dB_s\\
     &\quad    + \frac{1}{c} \int_0^{ct}\int_0^{\frac{1}{|Z^z_{s-|}}}\int_{-1}^1
                     Z^z_{s-}(u-1)
                       (\mathcal N-\mathcal N')(ds,dr,du),
      \qquad t\geq 0,
	\end{align*}
	almost surely. By Revuz and Yor \cite[Proposition V.1.5]{RY} we have the almost sure identity
 \[
   \frac{\sigma}{c} \int_0^{ct}\sqrt{|Z^z_s|} \,d B_s
       = \sigma\int_0^{t}\sqrt{c^{-1}|Z^z_{cs}|} d \big( c^{-1/2}B_{cs}\big),
         \qquad t\geq 0.
 \]
 Next, we use the analogue almost sure identity
  \begin{align*}
	 &\quad         \int_0^{ct}\int_0^{\frac{1}{|Z^z_{s-}|}}\int_{-1}^1 Z^z_{s-}(u-1) (\mathcal N-\mathcal N')(ds,dr,du)\\
            & = \int_0^t\int_0^{\frac{1}{\frac{1}{c}|Z^z_{(cs)-}|}}\int_{-1}^1Z^z_{(cs)-}(u-1)  (\mathcal N-\mathcal N')(c d s,c^{-1}dr,du),
               \qquad t\geq 0.
 \end{align*}
Motivated by the above two identities, we define the Wiener process
 $\bar B_t:=\frac{1}{\sqrt{c}}B_{ct}$, $t\geq 0$, and the independent Poisson random measures  $\bar\cN$ on $(0,\infty)\times (0,\infty)\times [-1,1]$
 by $\bar\cN(d s,d r,d u):=\cN(cd s,c^{-1}d r,d u)$.
 It follows directly from the definition of a Poisson random measure that
 $\bar{\mathcal N}$ is a Poisson random measure with the
 same intensity measure as $\mathcal N$.
With these definitions, the above calculation leads to
	\begin{align*}
     \bar Z^z_t   &= \frac{z}{c} +\bigg(\Psi(1)+\int_{-1}^0 (u-1) V(du)\bigg)\int_0^{t} \sign(\bar Z^z_s)\,ds
          +\sigma \int_0^{t} \sqrt{|\bar Z^z_{s}|} d \bar B_{s}\\
          &\quad+ \int_0^{t}\int_0^{\frac{1}{|\bar Z^z_{s-}|}}\int_{-1}^1\bar Z^z_{s-}(u-1)(\bar{\mathcal N} - \bar{\mathcal N}')(d s,d r,d u),
               \quad t\geq 0.
	\end{align*}
	Hence, $Z^{z/c}$ and $\bar Z^z$ satisfies the SDE \eqref{Equation} with initial condition $z/c$ and both do not spend time at zero and satisfies the symmetry condition \eqref{sym2}. But then the identity of one-dimensional marginals holds due to Step 1e) of the proof of Proposition \ref{Markov}. Finally, the Markov property proved in Proposition \ref{Markov} implies the identification of the finite-dimensional marginals and the self-similarity is proved.
	\smallskip
	
	The statement about the Lamperti-Kiu quintuple is a direct consequence of the construction of $Z^z$ via the SDE \eqref{Equation} and the definition of Lamperti-Kiu quintuples.
\end{proof}

\subsection{Proof of Theorem \ref{TheoremMain}}\label{ProofEnde}
Let us start with a simple reformulation of the Condition \eqref{cramer} in a restrictive setting:
	\begin{lemma}\label{L0}
		Suppose $(P^z)_{z\geq 0}$ is a positive self-similar Markov process that only jumps towards the origin. Then there is a unique self-similar extension $(P^z)_{z\geq 0}$ of $(P^z)_{z> 0}$ under which the canonical process leaves zero continuously precisely if  $\Psi(1)>0$.
	\end{lemma}
	\begin{proof}
		The spectrally negative assumption implies (for instance by the L\'evy-Khintchin representation) that $\Psi(\lambda)=\log \E[e^{\lambda \xi_1}]<\infty$ for all $\lambda\geq 0$. Note furthermore that, if well-defined, the Laplace exponent $\lambda\mapsto \Psi(\lambda)$ is a convex function on $\R_{\geq 0}$.
		\smallskip
		
		First, suppose $\xi$ drifts to $-\infty$, so that the existence of the claimed extension is equivalent to \eqref{cramer}. Since the right-derivative $\Psi'(0)$ equals $\E[\xi_1]<0$, there is some $s>0$ such that $\Psi(s)<0$. Thus, the convexity of $\lambda\mapsto \Psi(\lambda)$ implies that \eqref{cramer} is equivalent to $\Psi(1)>0$.
		\smallskip
		
		Next, suppose $\xi$ does not drift to $-\infty$. Then the existence of the claimed extension is equivalent to \eqref{over} which is trivially fulfilled since all overshoots are zero because all jumps are negative. At the same time the convexity of $\Psi$ and $\Psi'(0)=\E[\xi_1]\geq 0$ imply $\Psi(\lambda)>0$ for any $\lambda>0$. Hence, the claimed equivalence is trivial in this latter case.
	\end{proof}

To find a necessary condition for $\R_*$-valued self-similar Markov processes to have an extension that leaves zero continuously we want to apply Condition \eqref{cramer} for a suitably derived positive self-similar Markov process. Since we are only interested in symmetric self-similar processes the suitable choice is the absolute value.
	\begin{lemma}\label{L2}
			Suppose $(P^z)_{z\in\R}$ is a symmetric $\R_*$-valued self-similar Markov process with Lamperti-Kiu quintuple $(a,\sigma^2,\Pi,q,V)$ that satisfies Assumption \textbf{(A)} and define $(|P|^z)_{z\geq 0}$  as the law of $|Z|$ under $(P^z)_{z\in\R}$.
		\begin{itemize}
			\item[\textbf{(a)}]  $(|P|^z)_{z\geq 0}$   is a positive self-similar Markov process.
			\item[\textbf{(b)}] The Lamperti-transformed L\'evy process $\xi^{|P|}$ of $(|P|^z)_{z\geq 0}$ satisfies
			\begin{align*}
				\Psi^{|P|}(1)=\Psi(1)+\int_{-1 }^0 (|u|-1) V(du),
			\end{align*}
			where $\Psi$ is the Laplace exponent of the L\'evy process $\xi$ with triplet $(a,\sigma^2,\Pi)$ killed at rate $q$.
				
		\end{itemize}
		
	\end{lemma}
	\begin{proof}
		\textbf{(a)} The  Markov property for $(|P|^z)_{z\geq 0}$ is inherited from $(P^z)_{z\in\R}$ due to the symmetry assumption. The self-similarity carries over trivially. 
		\smallskip
		
 		\textbf{(b)} To determine $\Psi^{|P|}(1)$ we use Proposition \ref{PropositionMain2} twice. First recall from Proposition \ref{PropositionMain} that $P^z$ can be expressed as
		\begin{align*}
			Z_t&=z+\bigg(\Psi(1)+\int_{-1}^0 (u-1) V(du)\bigg)\int_0^t\sign(Z_s)\,ds+\sigma\int_0^t \sqrt{|Z_s|}\,dB_s\\
			&\quad+\int_0^t \int_0^{\frac{1}{|Z_{s-}|}}\int_{-1}^1 Z_{s-}(u-1)(\mathcal{N-N'})(ds,dr,du),\quad t\leq T_0.
	\end{align*}
	Taking absolute values as in the proof of Corollary \ref{Cor1} we find that $|P|^z$ is given by
		\begin{align*}
		|Z_t|&=\bigg(\Psi(1)+\int_{-1}^0 (|u|-1) V(du)\bigg)\,t+\sigma\int_0^t \sqrt{|Z_s|}\,d\bar B_s\\
			&\quad+\int_0^t\int_0^{\frac{1}{|Z_{s-}|}}\int_{-1}^1 |Z_{s-}|(|u|-1)(\mathcal{N-N'})(ds,dr,du),\quad t\leq T_0,
		\end{align*}
		with the Brownian motion $\bar B_t:=\int_0^t \sign(Z_s)\,dB_s$. Equivalently, we can write
		\begin{align}\label{44}
			\begin{split}
			|Z_t|
			&=\bigg(\Psi(1)+\int_{-1}^0 (|u|-1) V(du)\bigg)\,t+\sigma\int_0^t \sqrt{|Z_s|}\,d\bar B_s\\
			&\quad+\int_0^t\int_0^{\frac{1}{|Z_{s-}|}}\int_{0}^1 |Z_{s-}|(u-1)(\mathcal{\bar N-\bar N'})(ds,dr,du),\quad t\leq T_0,
		\end{split}
		\end{align}
		where $\bar{ \mathcal N}$ has intensity $ds\otimes dr\otimes \big(\bar\Pi(du)+V(-du)\big)$ on $(0,\infty)\times(0,\infty)\times(0,1)$. Comparing with \eqref{SDELeifMatyas} we can read off the L\'evy triplet for $\xi^{|P|}$ and in particular the Laplace exponent evaluated at $1$.
	\end{proof}

	We can now finish the proof of our main result.
	\begin{proof}[Proof of Theorem \ref{TheoremMain}]		
		Recall from Proposition \ref{PropositionMain} that the $\R_*$-valued self-similar symmetric Markov families obtained from real-valued self-similar Markov families by absorption at zero are completely characterized by Lamperti-Kiu quintuples $(a,\sigma^2,\Pi,q,V)$. \\
		To see that 			Condition \eqref{Condition} is necessary we apply Lemmas \ref{L0} and \ref{L2}: Suppose $(P^z)_{z\in\R}$ is a real-valued self-similar Markov process that leaves zero continuously. Then the Markov family $(|P^\dag|^z)_{z\geq 0}$ obtained by absorption at zero is a positive self-similar Markov family with a self-similar extension that leaves zero continuously. The Laplace exponent of the Lamperti transformed L\'evy process satisfies
			\begin{align*}
				\Psi^{|P^\dag|}(1)=\Psi(1)+\int_{-1}^0 (|u|-1) V(du)
			\end{align*}
			which, as we showed in Lemma \ref{L0} has to be strictly positive. 		
		\smallskip
		
	Conversely, if Condition \eqref{Condition} is satisfied for a given quintuple $(a,\sigma^2,\Pi,q,V)$, then by Proposition \ref{Pro2}, we constructed in Section \ref{SectionConstruction} a real-valued self-similar Markov process with Lamperti-Kiu triplet $(a,\sigma^2,\Pi,q,V)$. Furthermore, the solutions $Z^z$ leave zero continuously since the integrand of the Poissonian integral is zero at zero.
		\end{proof}
	
\section*{Acknowledgement}
	I would like to thank Jean Bertoin and Zenghu Li for interesting discussions on the subject and Matyas Barczy for careful reading of earlier manuscripts. Furthermore, I thank Beijing Normal University for very kind hospitality where part of this work was carried out.


\begin{thebibliography}{99}
 \bibitem{Aldous2}
 D. Aldous: {\sl Stopping times and tightness. I.} Ann. Probab. 2(6), 335--340 (1978)

\bibitem{App}
D. Applebaum: {\sl L\'evy Processes and Stochastic Calculus, 2nd edition.}
 Cambridge University Press, 2009.

\bibitem{BD}
M. Barczy, L. D\"oring:
On entire moments of self-similar Markov processes.
Submitted


\bibitem{BDMZ}
J. Berestycki, L. D\"oring, L. Mytnik and L. Zambotti:
On existence and uniqueness for self-similar SDEs driven by stable processes.
Prepring (2011).

\bibitem{BC}
J. Bertoin and M.-E. Caballero:
Entrance from $0+$ for increasing semi-stable Markov processes.
{\it Bernoulli} {\bf 8} 195--205, (2002).



\bibitem{BS}
J. Bertoin and M. Savov:
 Some applications of duality for L\'evy processes in a half-line.
 {\it Bull. Lond. Math. Soc.} {\bf 43} 97--110, (2011).

\bibitem{BY}
J. Bertoin and M. Yor:
The entrance laws of self-similar Markov processes and exponential functionals of L\'evy processes.
 {\it Potential Anal.} {\bf 17} 389--400, (2002).

\bibitem{BY2}
J. Bertoin and M. Yor:
On the entire moments of self-similar Markov processes and exponential functionals of L\'evy processes.
 {\it Annales de la facult\'e des sciences de Toulouse} {\bf S\'erie 6, Vol. 11} 33--45, (2002).

\bibitem{BY3}
J. Bertoin and M. Yor:
Exponential functionals of L\'evy processes.
 {\it Probab. Surv.} {\bf 2} 191--212, (2005).

\bibitem{Bl}
R. M. Blumenthal:
On construction of Markov processes.
{\it Z. Wahrscheinlichkeitstheorie Verw. Geb. }{\bf 63(4)} 433--444, (1983).



\bibitem{BBC}
R. Bass, K. Burdzy, Z.-Q. Chen:
Pathwise uniqueness for a degenerate stochastic differential equation.
{Annals of Probability {\bf 35}, 2385-2418, (2007).}

\bibitem{CC}
M.-E. Caballero and L. Chaumont:
Weak convergence of positive self-similar Markov processes and overshoots of L\'evy processes.
 {\it Ann. Probab.} {\bf 34} 1012--1034, (2006).



\bibitem{CPR}
L. Chaumont, H. Pant\'i, V. Rivero:
The Lamperti representation of real-valued self-similar Markov processes.
{\it Preprint}

\bibitem{CKPR}
L. Chaumont, A. Kyprianou, J. C. Pardo and V. Rivero:
Fluctuation theory and exit systems for positive self-similar Markov processes.
 To appear in {\it Ann. Probab.}


\bibitem{C}
O. Chybiryakov:
"The Lamperti correspondence extended to L\'evy processes and semi-stable Markov processes in locally compact groups."
Stochastic Processes and their Applications 116, pp. 41--55, (2006).



\bibitem{DM}
C. Dellacherie and P. A. Meyer:
"Probabilit\'es et potentiel: Chapitres V  a VIII Th\'eorie des martingales."
Hermann, Paris, 1983.

\bibitem{DB}
L. D\"oring and M. Barczy:
A Jump Type SDE Approach to Positive Self-Similar Markov Processes.
{\it Preprint}


\bibitem{EthierKurtz}
S.N. Ethier and T.G. Kurtz: {\sl Markov processes: characterization and convergence.}
 Wiley, 1986.

\bibitem{F}  P. Fitzsimmons: On the existence of recurrent extensions of self-similar
 Markov processes. {\it Electron. Comm. Probab.} {\bf 11} 230-241, (2006).

\bibitem{FL}
Z. Fu and Z.H. Li:
Stochastic equations of non-negative processes with jumps.
 {\it Stochastic Process. Appl.} {\bf 120} 306--330, (2010).


\bibitem{IW}
N. Ikeda and S. Watanabe:
"Stochastic differential equations and diffusion processes."
 North-Holland Publishing Company, 1981.


\bibitem{JS}
J. Jacod and A.N. Shiryaev: {\sl Limit theorems for stochastic processes. Second edition.}
 Springer-Verlag, Berlin, 2003.

\bibitem{KS}
I. Karatzas and S.E. Shreve: {\sl Brownian motion and stochastic calculus. Second edition.}
 Springer-Verlag, New York, 1991.


\bibitem{K1}
S.W. Kiu:
Two dimensional semi-stable Markov processes.
{\it Ann. Probab.} 3 (3), 440--448, (1975)

\bibitem{K2}
S.W. Kiu:
Semistable Markov processes in $\R^n$.
{\it Stochastic Process. Appl} 10 (2), 205-225, (1980)

\bibitem{KP}
A. Kutznetsov, J.C. Pardo:
Fluctuations of stable processes and exponential functionals of hypergeometric Levy processes.
{arXiv:1012.0817 }


\bibitem{L}
J. Lamperti:
"Semi-stable Markov processes. I."
 {\it Z. Wahr. und Verw. Gebiete} {\bf 22} 205--225, (1972).

\bibitem{LP}
Z. Li and F. Pu
"Strong solutions of jump-type stochastic equations" (2012)
arxiv:1205.1085

 \bibitem{P}
P. Patie:
Exponential functional of a new family of L\'evy processes and self-similar continuous state branching processes with immigration.
{\it Bull. Sci. Math.} {\bf 133}(4) 355--382, (2009).

\bibitem{Protter}
P. Protter:
Stochastic integration and differential equations. Second edition.
{\it Springer-Verlag, 2004}


\bibitem{RY} D. Revuz and M. Yor:
 {\sl Continuous martingales and Brownian motion. Third edition.}
 Springer-Verlag Berlin Heidelberg, 1999.

\bibitem{R1}
V. Rivero: Recurrent extensions of self-similar Markov processes and Cram\'er's condition.
 {\it Bernoulli} {\bf 11}(3) 471--509 (2005).

\bibitem{R2}
V. Rivero: Recurrent extensions of self-similar Markov processes and Cram\'er's condition. II.
 {\it Bernoulli} {\bf 13}(4) 1053--1070 (2007).

\bibitem{Sato}
K.-I. Sato: {\sl L\'evy processes and infinitely divisible distributions.}
 Cambridge University Press, Cambridge, 1999.

\bibitem{YW1}
S. Yamada and T. Watanabe:
 On the uniqueness of solutions of stochastic differential equations.
 {\it J. Math. Kyoto Univ.} {\bf 11} 155--167, (1971).

\end{thebibliography}
\end{document}